\documentclass[a4paper,11pt]{article} 
\title{Splitting formulas for the logarithmic double ramification cycle}

\usepackage{amsmath, amssymb, mathrsfs, amsthm, shorttoc, geometry, stmaryrd}
\usepackage{tikz}
\usepackage{mathtools} 
\usepackage{hyperref}
\usepackage{xcolor}
\usepackage{subcaption}
\hypersetup{
colorlinks,
linkcolor={black},
citecolor={blue!50!black},
urlcolor={blue!80!black}
}
\usepackage[all]{xy}
\usepackage{tabularx}
\usepackage{longtable}
\numberwithin{equation}{subsection}
\usepackage{enumitem}
\usepackage{mathtools}

\usepackage{cleveref}
\crefformat{equation}{(#2#1#3)}
\let\oref\ref
\let\tilde\widetilde
\AtBeginDocument{\renewcommand{\ref}[1]{\Cref{#1}}}

\usepackage{pgf,tikz}
\usetikzlibrary{matrix, calc, arrows, math}

\usetikzlibrary{decorations.pathmorphing, knots}

\usepackage{tikz-cd} 

\usepackage{todonotes}
\reversemarginpar

\makeatletter
\newcommand*{\doublerightarrow}[2]{\mathrel{
  \settowidth{\@tempdima}{$\scriptstyle#1$}
  \settowidth{\@tempdimb}{$\scriptstyle#2$}
  \ifdim\@tempdimb>\@tempdima \@tempdima=\@tempdimb\fi
  \mathop{\vcenter{
    \offinterlineskip\ialign{\hbox to\dimexpr\@tempdima+1em{##}\cr
    \rightarrowfill\cr\noalign{\kern.5ex}
    \rightarrowfill\cr}}}\limits^{\!#1}_{\!#2}}}
\newcommand*{\triplerightarrow}[1]{\mathrel{
  \settowidth{\@tempdima}{$\scriptstyle#1$}
  \mathop{\vcenter{
    \offinterlineskip\ialign{\hbox to\dimexpr\@tempdima+1em{##}\cr
    \rightarrowfill\cr\noalign{\kern.5ex}
    \rightarrowfill\cr\noalign{\kern.5ex}
    \rightarrowfill\cr}}}\limits^{\!#1}}}
\makeatother

\newcommand{\ca}[1]{{\mathcal{#1}}}

\newcommand{\ul}[1]{{\underline{#1}}}
\newcommand{\ol}[1]{{\overline{#1}}}

\newcommand{\Mbar}{\overline{\ca M}}
\newcommand{\Mpt}{\mathbb{M}}

\newcommand{\DRL}{\operatorname{DRL}}
\newcommand{\DR}{\operatorname{DR}}
\newcommand{\gp}{\operatorname{gp}}

\newcommand{\isom}{\stackrel{\sim}{\longrightarrow}}

\renewcommand{\gp}{{\operatorname{gp}}}

\newcommand{\tensor}{\otimes}

\newcommand{\LogSch}{\textbf{LogSch}}

\newcommand{\PL}{\text{PL}}

\newcommand{\res}{\text{res}}
\newcommand{\trop}{\mathrm{trop}}

\newcommand{\Jac}{\text{Jac}}

\let\div\relax

\DeclareMathOperator{\div}{div}

\DeclareMathOperator{\Hom}{Hom}

\DeclareMathOperator{\Aut}{Aut}

\DeclareMathOperator{\ev}{ev}

\DeclareMathOperator{\CH}{CH}
\DeclareMathOperator{\LogCH}{LogCH}

\DeclareMathOperator{\LogDR}{LogDR}
\DeclareMathOperator{\LLogDR}{\mathbb{L}ogDR}
\DeclareMathOperator{\LogDRL}{LogDRL}
\DeclareMathOperator{\LLogDRL}{\mathbb{L}ogDRL}
\DeclareMathOperator{\LogJ}{LogJac}
\DeclareMathOperator{\LLogJ}{\mathbb{L}ogJac}

\DeclareMathOperator{\lcm}{lcm}

\newcommand{\A}{\mathbb{A}}
\newcommand{\B}{\text{B}}

\newcommand{\G}{\mathbb{G}}

\newcommand{\N}{\mathbb{N}}
\renewcommand{\P}{\mathbb{P}}
\newcommand{\Q}{\mathbb{Q}}
\newcommand{\R}{\mathbb{R}}

\newcommand{\Z}{\mathbb{Z}}

\newcommand{\Lcal}{\mathcal{L}}
\newcommand{\Mcal}{\mathcal{M}}

\newcommand{\Ocal}{\mathcal{O}}

\renewcommand{\log}{\mathrm{log}}
\newcommand{\pt}{\mathsf{pt}}

\newcommand{\gl}{\mathrm{gl}}
\newcommand{\fs}{\mathrm{fs}}
\newcommand{\qs}{\mathrm{qs}}

\newcommand{\M}{{\mathsf{M}}}
\newcommand{\ghost}{\overline{{\mathsf {M}}}}

\newcommand{\pie}{\mathring}

\newcommand{\vfc}{\mathrm{vfc}}

\tikzset{
    labl/.style={anchor=south, rotate=-90, inner sep=.5mm}
}

\newcommand*{\pb}{\ar[dr, phantom, very near start, "{\ulcorner}"]}

\theoremstyle{definition}
\newtheorem{definition}{Definition}[section]

\newtheorem{example}[definition]{Example}
\newtheorem{remark}[definition]{Remark}

\theoremstyle{plain}

\newtheorem{proposition}[definition]{Proposition}
\newtheorem{lemma}[definition]{Lemma}
\newtheorem{theorem}[definition]{Theorem}
\newtheorem{corollary}[definition]{Corollary}

\newtheorem{intheorem}{Theorem}

\newtheorem{inproposition}[intheorem]{Proposition}

\theoremstyle{remark}

\usepackage{letltxmacro}
\LetLtxMacro{\phiorig}{\phi}
\renewcommand{\phi}{\varphi}
\renewcommand{\hat}{\widehat}
\renewcommand{\tilde}{\widetilde}

\author{Pim Spelier}
\date{\today}

\newcommand{\beq}{\begin{equation}}
\newcommand{\eeq}{\end{equation}}
\newcommand{\beqs}{\begin{equation*}}
\newcommand{\eeqs}{\end{equation*}}

\tikzset{
  symbol/.style={
    draw=none,
    every to/.append style={
      edge node={node [sloped, allow upside down, auto=false]{$#1$}}}
  }
}

\newcommand{\colim}{\operatornamewithlimits{colim}}

\begin{document}
\maketitle

\begin{abstract}
The logarithmic double ramification cycle is roughly a logarithmic Gromov--Witten invariant of $\mathbb{P}^1$. For classical Gromov--Witten invariants, formulas for the pullback along the gluing maps have been invaluable to the theory. For logarithmic Gromov--Witten invariants, such formulas have not yet been found. One issue is the fact that log stable maps cannot be glued. In this paper, we use the framework from \cite{Holmes2023LogarithmicCohomologicalFT} for gluing pierced log curves (a refinement of classical log curves) to give formulas for the pullback of the (log) (twisted) double ramification cycle along the loop gluing map.
\end{abstract}


\section{Introduction}
In this paper we present splitting formulas for the pullback of the \emph{(twisted) double ramification cycle} and the \emph{logarithmic (twisted) double ramification cycle} along the gluing map
\begin{align*}
\Mbar_{g-1,n+2} & \to \Mbar_{g,n}. 
\end{align*}

The logarithmic double ramification cycle is a logarithmic Gromov--Witten invariant of $[\P^1/\G_m]$, and is related to log Gromov--Witten invariants of toric varieties \cite{Ranganathan2023Logarithmic-Gromov-Witten} and Hurwitz numbers of $\P^1$ \cite{DoubleHurwitz2025CMR,cavalieri2024kleakydoublehurwitzdescendants}. The twisted (log) double ramification cycles are generalisations relevant to surface dynamics \cite{ChenTeichmullerDynamics} and computing strata of differentials and volumes of moduli spaces \cite{farkas24stratadifferentialsdoubleramification,SchmittDimensionTheory,SauvagetCohomologyClassesStrataOfDifferentials}. The splitting formulas for classical Gromov--Witten invariants \cite{Behrend1997GromovWittenAlgebraicGeometry} have been invaluable to enumerative geometry but have not yet seen logarithmic equivalents. Difficulties arise from the gluing map not respecting the log structure: in other words, one cannot glue log curves or log stable maps, nor pullback logarithmic classes.

For the double ramification cycle, a certain projection of the logarithmic double ramification cycle, a pullback formula along the separating gluing map $\Mbar_{g_1,n_1+1} \times \Mbar_{g_2,n_2+1} \to \Mbar_{g,n}$ was given in \cite{Buryak2019Quadratic-doubl}. In \cite{Holmes2023LogarithmicCohomologicalFT} a logarithmic enhancement of log curves as given, allowing for gluing of log curves and log stable maps. That framework was used to lift the result from \cite{Buryak2019Quadratic-doubl} to the logarithmic double ramification cycle. A formula for the pullback of the logarithmic double ramfiication cycle along the non-separating gluing map $\Mbar_{g-1,n+2} \to \Mbar_{g,n}$ was still open.

Here we use the framework from \cite{Holmes2023LogarithmicCohomologicalFT} to give a formula for the pullback of the logarithmic double ramification cycle along the non-separating gluing map. We use this to additionally compute a formula for the pullback of the classical double ramification cycle. This latter formula in the untwisted case is originally due to Zvonkine, in unpublished work \cite{zvonkine}, and was of great help in finding the logarithmic formula.  

We first give an introduction to the logarithmic double ramification cycle and logarithmic gluing maps. In \ref{insec:split} we explain the splitting formula, and the geometric story behind it.

\subsection{The (logarithmic) double ramification cycle.} Fix $g,n$ and let $A \in \Z^n$ be a sequence with sum $0$. Then the \emph{smooth double ramification locus} is the closed substack of $\Mcal_{g,n}$ defined as
\[
  \DRL(A)^\circ = \{C \in \Mcal_{g,n} : \Ocal_C(\sum_i a_i p_i) \cong \Ocal_C\} \subset \Mcal_{g,n}.
\]
The line bundle $\Ocal_C(\sum_i a_i p_i)$ is trivial if and only if there is a function $f: C \to \P^1$ with tangencies to $0$ and $\infty$ prescribed by the $a_i$. This leads to the natural compactification $\Mbar_{A}^{\sim}$ 
of stable maps to \emph{rubber}, stable maps to (chains of) $\P^1$ with tangency conditions given by $a_i$ \cite{Graber2003Hodge-integrals,Li2001Stable-morphism,Li2002A-degeneration-}. The pushforward of the (virtual) fundamental class of $\Mbar_A^\sim$ to $\Mbar_{g,n}$ is the \emph{double ramification cycle}
\[
\DR(A) \in \CH^g(\Mbar_{g,n})
\]
A formula for this class has been given in \cite{Janda2016Double-ramifica} (later generalised by \cite{Bae2020Pixtons-formula}).

The smooth double ramification locus also admits the description as the pullback of the Abel--Jacobi map $\Mcal_{g,n} \to \Jac_{g,n}, C \mapsto \Ocal(\sum_i a_i p_i)$ along the zero section $\Mcal \to \Jac_{g,n}$.

The Abel--Jacobi map does not extend to the boundary of $\Mbar_{g,n}$. In \cite{Holmes2017Extending-the-d} they construct a cycle $\hat{\DR}(A)$ on a log blowup\footnote{A log blowup of $\Mbar_{g,n}$ is an (iterated) blowup in boundary strata.} $\hat{\Mcal}_x \to \Mbar_{g,n}$ that (partially) resolves the Abel--Jacobi map, and then obtain $\DR(A)$ as the pushforward of $\hat{\DR}(A)$.

It turns out this object $\hat{\DR}(A)$ is in many ways more natural and better behaved. The \emph{logarithmic Chow ring} is the ring \[\LogCH^*(\Mbar_{g,n}) = \colim_{\tilde{\Mcal} \to \Mbar_{g,n}} \CH^*(\tilde{\Mcal}_{g,n})\] where the colimit is taken over all log blowups. 

In \cite{Holmes2021Logarithmic-int} they define the \emph{logarithmic double ramification cycle}
\[
  \LogDR(A) \in \LogCH^g(\Mbar_{g,n})
\]
as the image of $\hat{\DR}(A)$ under the inclusion $\CH^g(\hat{\Mcal}_A) \to \LogCH^g(\Mbar_{g,n})$.

This logarithmic refinement has played an important role. For example, recent work has shown one can compute logarithmic Gromov--Witten invariants of toric varieties and double Hurwitz numbers in terms of logarithmic double ramification cycles, but not in terms of the classical double ramification cycle \cite{Ranganathan2023Logarithmic-Gromov-Witten,DoubleHurwitz2025CMR,cavalieri2024kleakydoublehurwitzdescendants}. Furthermore, the logarithmic double ramification cycle satisfies multiplicative properties the double ramification cycle does not \cite{Holmes2017Multiplicativit,Holmes2021Logarithmic-int}.

This case has also been generalised to the \emph{twisted} (logarithmic) double ramification cycles $\LogDR(A)$ and $\DR(A)$ for a sequence $A \in \Z^n$ with sum $k(2g-2+n)$ for a (necessarily unique) $k \in \Z$, measuring the locus where the line bundle $\omega_{C}^{\log,\tensor -k}(\sum_i a_i p_i)$ is trivial. These cycles relate to the strata of differentials $[\mathcal{H}_g(A)]$, and this connection has been used to compute virtual classes and intersection theory of strata of differentials \cite{farkas24stratadifferentialsdoubleramification,SchmittDimensionTheory,SauvagetCohomologyClassesStrataOfDifferentials}. See \cite{chen22ataleoftwomodulispaces} for an overview of the connection between the logarithmic double ramification cycle and multiscale differentials.

A formula for the logarithmic double ramification cycle has been given in \cite{Holmes2022Logarithmic-double}.

\subsection{(Logarithmic) gluing maps.}
The gluing maps have played a central role in the enumerative geometry of curves. The splitting formulas for Gromov--Witten invariants, first shown in \cite{Behrend1997GromovWittenAlgebraicGeometry}, have been instrumental, and by the Givental--Teleman classification sometimes even allow for the reconstruction of all Gromov--Witten invariants \cite{Kontsevich1994Gromov-Witten-c,Pandharipande2018CohFTCalculations}.

The degeneration formula \cite{abramovich2020decomposition,kim2018degeneration,ranganathan2019logarithmic} relates Gromov--Witten invariants to \emph{logarithmic} Gromov--Witten invariants, counting logarithmic curves with additional tangency conditions \cite{Li2001Stable-morphism,Li2002A-degeneration-,Gross2013Logarithmic-gro}.
For logarithmic Gromov--Witten invariants, splitting formulas are more difficult to obtain, and several issues pop up.

An important issue is the failure of gluing log curves. 
\begin{center}
\emph{The gluing maps are not maps of log stacks, only a map of underlying algebraic stacks.}
\end{center}
The gluing maps send the locus of smooth curves into the boundary, which is not allowed for a map of log stacks. Because of this, there is simply no possible way to glue logarithmic curves. It follows one cannot glue log stable maps. And since log blowups of $\Mbar_{g,n}$ do not pull back to log blowups of $\Mbar_{g_1,n_1+1} \times \Mbar_{g_2,n_2+1}$ or $\Mbar_{g-1,n+2}$, there is no pullback map on $\LogCH$.

This issue has a resolution in \cite{Holmes2023LogarithmicCohomologicalFT}, where logarithmic enhancements $\pi: \Mpt_{g,n} \to \Mbar_{g,n}$ are given, parametrising \emph{pierced} curves (see \ref{sec:pierceddr} for a summary). The map $\pi$ is an isomorphism on underlying stacks, but $\Mpt_{g,n}$ has a different, generically non-trivial, log structure, and a universal pierced curve. These pierced curves have logarithmic sections, and by taking coproducts one obtains gluing maps 
\begin{align*}
\Mpt_{g_1,n_1+1} \times \Mpt_{g_2,n_2+1} &\to \Mpt_{g,n} \\
\Mpt_{g-1,n+2} & \to \Mpt_{g,n} 
\end{align*}
that induce pullback maps
\begin{align*}
\LogCH(\Mpt_{g,n}) &\to \LogCH(\Mpt_{g_1,n_1+1} \times \Mpt_{g_2,n_2+1})  \\
\LogCH(\Mpt_{g,n}) &\to \LogCH(\Mpt_{g-1,n+2}). 
\end{align*}

Tropically, pierced curves correspond to metrised graphs where legs also have lengths. Gluing two legs then produces an edge of length the sum of the two leg lengths.

We can now define a further refinement of the logarithmic double ramification cycle, namely the \emph{pierced} logarithmic double ramification cycle.

\begin{definition}[{\ref{def:pierceddr}}]
Let $\pi: \Mpt_{g,n} \to \Mbar_{g,n}$ be the forgetful map. For $A \in \Z^n$ with $\sum_i a_i = k(2g-2+n)$, define the \emph{pierced} logarithmic double ramification cycle
\[
  \LLogDR(A) = \pi^* \LogDR(A).
\]
\end{definition}

In this paper, we use this framework to give a splitting formula for this further refinement of the logarithmic double ramification cycle. 

By the universal property of gluing proven in \cite{Holmes2023LogarithmicCohomologicalFT}, one also obtains gluing maps
\begin{align*}
\Mpt_{g_1,n_1+1}(X) \times_X \Mpt_{g_2,n_2+1}(X) &\to \Mpt_{g,n}(X) \\
\Mpt_{g-1,n+2}(X) \times_{X \times X} X & \to \Mpt_{g,n}(X). 
\end{align*}

In future work \cite{2025HerrHolmesSpelierBirationalityAndGluingGromovWitten} this will be used to prove a splitting formula for logarithmic Gromov--Witten invariants.

\subsection{The splitting formula}
\label{insec:split}
Our main result is the logarithmic splitting formula \ref{inthm:firstformula}, a recursive formula in terms of smaller double ramification cycles and piecewise polynomials.

To clarify the formula, we explain the recursive structure of the double ramification locus on the boundary. See \ref{fig:m12,fig:mbar12} for two pictures corresponding to the classical case and the logarithmic case. We first treat the classical double ramification locus
\[
  \DRL_{g+1}(A) = \{(C,p_1,\dots,p_{n-2}) : \exists f: C \to \P^1, \div f = \sum_{i=1}^{n-2} a_i p_i\} \subset \Mbar_{g+1,n-2}
\]
for a fixed $A \in \Z^{n-2}$ with sum $0$. This is \emph{not} a closed substack of $\Mbar_{g+1,n-2}$.

\tikzmath{\dy1 = -3; \xA = 1.3; \arsize = 2mm;}

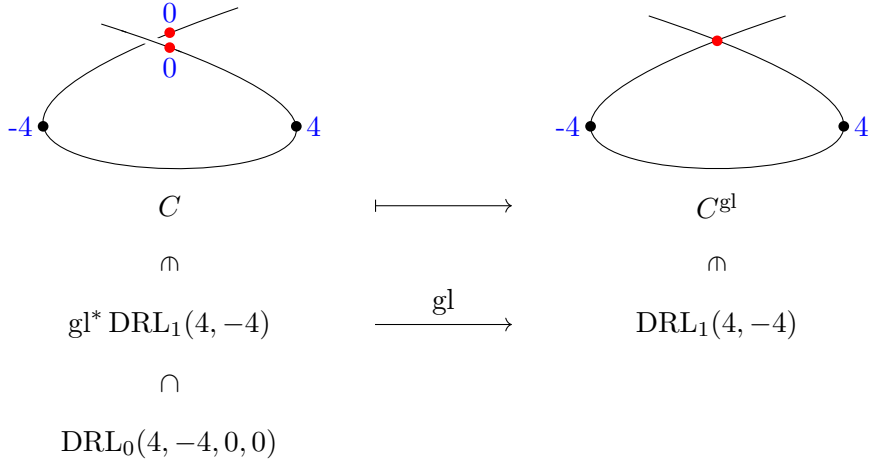
\begin{figure}
\centering
\begin{tikzpicture}[scale=.9,xshift=-70]

        \useasboundingbox (-10,0) rectangle (2,-6.2);

        \begin{knot}[clip width=5,yshift=-20]
          \strand (-1,1) .. controls (8,-2) and (-8,-2) .. (1,1);
        \end{knot}

        \begin{scope}[every node/.style={circle,fill,inner sep=0pt, minimum width=4pt}]
    \node[fill=red] (A) at (0,-0.07) {};
    \node at (-1.85,-1.33) {};
    \node at (1.85,-1.33) {};
    \end{scope}
        \node at (-1.85,-1.33) [left,blue] {-4};
    \node at (1.85,-1.33) [right,blue] {4};

        \node at (0,-2.5) {$C^\gl$};
        \node at (-8,-2.5) {$C$};

        \node [rotate = -90] at (0,-3.35) {$\in$};
        \node [rotate = -90] at (-8,-3.35) {$\in$};

        \node at (0,-4.25) {$\DRL_1(4,-4)$};

        \node [rotate = -90] at (-8,-5.12) {$\subset$};

        \node at (-8,-6) {$\DRL_0(4,-4,0,0)$};

        \node at (-8,-4.25) {$\gl^*\DRL_1(4,-4)$};

    \draw[|->] (-8+3,-2.5) --  (-3,-2.5);
    \draw[->] (-8+3,-4.25) -- node[midway,above] {$\gl$} (-3,-4.25);

    \begin{scope}[shift=({-8,0})]
    \begin{knot}[clip width=3, consider self intersections, ignore endpoint intersections=false, yshift=-20]
          \strand (-1,.88) .. controls (8,-2) and (-8,-2) .. (1,1.12);
        \end{knot}

      \begin{scope}[every node/.style={circle,fill,inner sep=0pt, minimum width=4pt}]
    \node[fill=red] at (0,-0.17) {};
    \node[fill=red] at (0,0.05) {};
    \node at (-1.85,-1.33) {};
    \node at (1.85,-1.33) {};
    \end{scope}
    \node at (0,-0.17) [below,blue] {0};
    \node at (0,0.05) [above,blue] {0};
        \node at (-1.85,-1.33) [left,blue] {-4};
    \node at (1.85,-1.33) [right,blue] {4};
    \end{scope}
\end{tikzpicture}
\caption{A curve $C$ in $\gl^*\DRL_1{(4,-4)} \subset \Mbar_{0,4}$. The glued curve $C^\gl \in \DRL_1(4,-4) \subset \Mbar_{1,2}$ admits a rational function $f$ with divisor $\div f = 4p_1 - 4p_2$ shown in blue. The pulback $\gl^*f$ on $C$ is a rational function with divisor $4p_1 - 4p_2$, hence we have $C \in \DRL_0(4,-4,0,0)$.}
\label{fig:m12}
\end{figure}

Let $\gl^*\DRL_{g+1}(A)$ denote the pullback along $\gl: \Mbar_{g,n} \to \Mbar_{g+1,n-2}$. Take $C \in \gl^*\DRL_{g+1}(A)$, a curve of genus $g$ with $n$ markings $p_1,\dots,p_n$, and let $\gl: C \to C^\gl$ be the gluing of $p_{n-1}$ and $p_n$. By definition of $\gl^* \DRL(A)$, on $C^\gl$ there is a rational function $f: C^\gl \to \P^1$ on $C^\gl$ with \[\div f = \sum_{i=1}^{n-2} a_i p_i.\] Then $f_0 = \gl^* f$ is a rational function $f_0: C \to \P^1$ with $\div f_0 = \sum_{i=1}^{n-2} a_i p_i$ and $f_0(p_{n-1}) = f_0(p_n) \ne 0$. We find \begin{align*}\gl^* \DRL_{g+1}(A) &= \{(C,p_1,\dots,p_n) \in \Mbar_{g,n} : \exists f: C \to \P^1, \div f = \sum_i a_i p_i \text{ and } f(p_{n-1}) = f(p_n)\}\\ &\subset \DRL_g(A,0,0)\end{align*} or equivalently, we obtain a pullback diagram
\begin{equation}
\label{diag:dr}
\begin{tikzcd}[contains/.style = {draw=none,"\ni" description,sloped}]
\gl^* \DRL_{g+1}(A) \arrow[r]\arrow[d] & \DRL_g(A,0,0) \arrow[d] \arrow[r,contains] &\arrow[d,mapsto] (C,p_1,\dots,p_n, f: C \to \P^1) \\
1 \arrow[r] & \G_m \arrow[r,contains] & f(p_{n-1})/f(p_n)
\end{tikzcd}
\end{equation}

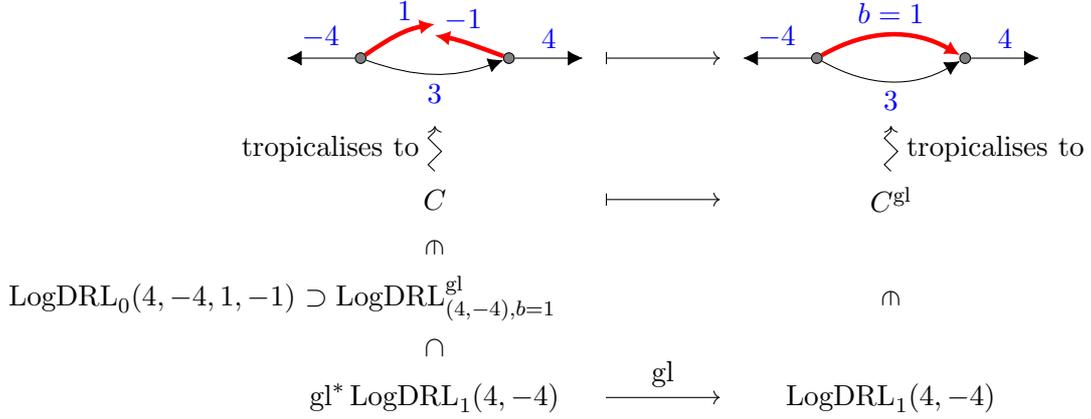
\begin{figure}
\centering
\begin{tikzpicture}[scale=.75]
    \begin{scope}[every node/.style={circle, draw,fill=black!50,inner sep=0pt, minimum width=4pt}]
    \node (A) at (-\xA,0) {};
    \node (B) at (\xA,0) {};

    \end{scope}
        \draw[line width = 1.7, red, -{Latex[length=\arsize, width=\arsize]}] (A) to [in=180-30,out=30] node [midway, above, blue] {$b = 1$} (B);
        \draw[-{Latex[length=\arsize, width=\arsize]}] (A) to [in=180+30,out=-30] node [midway, below, blue] {$3$} (B);
        \draw[-{Latex[length=\arsize, width=\arsize]}] (A) -- node [midway, above, blue] {$-4$} (-2*\xA,0);
        \draw[-{Latex[length=\arsize, width=\arsize]}] (B) -- node [midway, above, blue] {$4$} (2*\xA,0);

        \node at (0,-2.5) {$C^\gl$};
        \node at (-8,-2.5) {$C$};

        \node [rotate = -90] at (0,-4.25) {$\in$};
        \node [rotate = -90] at (-8,-3.35) {$\in$};

        \node at (0,-6) {$\LogDRL_1(4,-4)$};

        \node [rotate = -90] at (-8,-5.12) {$\subset$};

        \node at (-8,-6) {$\gl^* \LogDRL_1(4,-4)$};

        \node [xshift = -57] at (-8,-4.25) {$\LogDRL_0(4,-4,1,-1) \supset \LogDRL_{(4,-4),b=1}^\gl$};
    \draw[->, decorate, decoration={zigzag, post=lineto,post length=2pt}] (0,-2) -- node [midway, right, xshift = 2] {tropicalises to} (0,-1.2);
    \draw[->, decorate, decoration={zigzag, post=lineto,post length=2pt}] (-8,-2) -- node [midway, left, xshift = -2] {tropicalises to} (-8,-1.2);

    \draw[|->] (-8+3,0) --  (-3,0);
    \draw[|->] (-8+3,-2.5) -- (-3,-2.5);
     \draw[->] (-8+3,-6) -- node[midway,above] {$\gl$} (-3,-6);

    \begin{scope}[shift=({-8,0})]
    \begin{scope}[every node/.style={circle, draw,fill=black!50,inner sep=0pt, minimum width=4pt}]
    \node (C) at (-\xA,0) {};
    \node (D) at (\xA,0) {};
    \end{scope}
        \draw [line width = 1.7, red, -{Latex[length=\arsize, width=\arsize]}] (C) to [out=35,in = 190] node [pos=.6, above, blue] {$1$} (0,.6);
        \draw [line width = 1.7, red, -{Latex[length=\arsize, width=\arsize]}] (D) to [out=180-20,in = -10] node [pos=.6, above, blue] {$-1$} (0,.4);
        \draw[-{Latex[length=\arsize, width=\arsize]}] (C) to [in=180+20,out=-20] node [midway, below, blue] {$3$} (D);
        \draw[-{Latex[length=\arsize, width=\arsize]}] (C) -- node [midway, above, blue] {$-4$} (-2*\xA,0);
        \draw[-{Latex[length=\arsize, width=\arsize]}] (D) -- node [midway, above, blue] {$4$} (2*\xA,0);
    \end{scope}
\end{tikzpicture}
\caption{A logarithmic curve $C$ in $\LogDRL_{(4,-4),b=1}^\gl \subset \gl^*\DRL_1{(4,-4)} \subset \Mbar_{0,4}$ and its tropicalisation $C^\trop$. The glued curve $C^\gl \in \DRL_1(4,-4) \subset \Mbar_{1,2}$ admits a map $f: C^\gl \to \G_\log$ whose tropicalisation is the piecewise linear function on $C^{\gl,\trop}$ with slopes shown in blue. The pulback $\gl^*f$ on $C$ is a rational function whose tropicalisation has slopes $(4,-4,1,-1)$ along the legs, hence we have $C \in \DRL_0(4,-4,1,-1)$.}
\label{fig:mbar12}
\end{figure}

For the logarithmic double ramification locus $\LLogDRL_{g+1}(A)$ (\ref{def:pierceddr}), a compactification of $\DRL_{g+1}(A)$, there is an additional phenomenon of tangencies at the new node. For $C^\gl \in \LLogDRL_{g+1}(A)$ the gluing of a pierced log curve $C \in \Mpt_{g,n}$, there is a piecewise linear function $\alpha_A: \Gamma^{\gl}$ on the metrised dual graph $\Gamma^\gl$ with slopes given by $a_i$ on the leg $p_i$, and total sum of slopes $0$ at every vertex. See \ref{fig:mbar12} for an example with $g = 0, n = 4, A = (4,-4)$. Write $b \in \Z$ for the slope along the glued edge $e$ (oriented from $p_{n-1}$ to $p_n$). In \ref{sec:pullbacklogDR} we prove we then have $C \in \LLogDRL_g(A,b,-b)$. In \ref{def:complog} we show we have a splitting based on slopes\footnote{By properness, we have $\LLogDRL_{A,b}^\gl = \varnothing$ for $|b| \gg 0$. Explicitly, we show this holds for $|b| > \sum_{i : a_i > 0} a_i$ in \ref{lem:bounded}.}
\[
  \gl^* \LLogDRL_{g+1}(A) = \bigsqcup_{b \in \Z} \LLogDRL_{A,b}^\gl
\]
with
\[
  \LLogDRL_{A,b}^\gl \subset \LLogDRL_g(A,b,-b).
\]
To generalise \ref{diag:dr}, we use that the sections $p_1,\dots,p_n: S \to C$ of a pierced log curve $C/S$ are logarithmic, and hence we obtain evaluation maps (this unfortunately fails for punctured or classical log curves). This allows us to write down the following pullback diagram (\ref{prop:glogpullback})
\begin{equation}
\label{diag:logdr}
\begin{tikzcd}[contains/.style = {draw=none,"\ni" description,sloped}]
\LLogDRL_{A,b}^\gl \arrow[r]\arrow[d] & \LogDRL_g(A,b,-b) \arrow[d, "\ev_{n-1}/\ev_n"] \arrow[r,contains] &\arrow[d,mapsto] (C,p_1,\dots,p_n, f) \\
1 \arrow[r] & \G_\log \arrow[r,contains] & f(p_{n-1})/f(p_n)
\end{tikzcd}
\end{equation}
Here $\G_\log$ is the logarithmic multiplicative group, a proper group space compactifying $\G_m$ (see \ref{ex:glog}) that admits a blowup $\P^1 \to \G_\log$, and $\ev_{n-1},\ev_n$ are evaluation maps (see \ref{def:evmaps}). Recall that there is a map
\[
  \Phi: \PL(X^\trop) \to \LogCH(X) 
\]
from piecewise linear functions on the tropicalisation of $X$ to logarithmic Chow (see \ref{sub:logarithmic_chow}). In $\P^1$ we have $[1] = [0] = [\infty] \in \CH(\P^1)$, and hence $[1] = \Phi(\max(x,0)) = \Phi(\max(-x,0))$ for $x$ the linear function given by the inclusion $\Sigma_{\P^1} \to \R$.

In \ref{sec:pullbacklogDR} we extend this geometric story to $\Mpt_{g,n}$, and check the virtual fundamental classes on the different logarithmic loci are compatible with the Gysin pullback along $1 \to \G_\log$ in \ref{diag:logdr}. We obtain the following logarithmic splitting formula. The formula depends on a small non-degenerate stability condition $\theta$, as defined in \cite{Kass2017The-stability-s}. The formula is expressed in terms of explicit piecewise linear functions $\delta_{B}^\theta$ on the cone stack $\Mpt_{g,n}^\trop$ (\ref{def:alphabtheta}). On $\LLogDRL_g(A,b,-b)$ this piecewise linear function is independent of $\theta$, and is the tropicalisation of the map $\ev_{n-1}/\ev_n$ in \ref{diag:logdr}. 

\begin{intheorem}[\ref{thm:firstformula}]
\label{inthm:firstformula}
Let $\theta$ be a small non-degenerate stability condition.

Let $\gl$ be the map $\Mpt_{g,n} \to \Mpt_{g+1,n-2}$. Let $A \in \Z^{n-2}$ with sum $k(2g-2+n)$. For $b \in \Z$, let $\delta_{A,b,-b}^\theta \in \PL(\Mpt_{g,n})$ denote the piecewise linear function from \ref{def:alphabtheta}. We have a splitting 
\begin{equation}
\gl^* \LLogDR(A) = \sum_{b \in \Z} \LLogDR_{A,b}^\gl
\end{equation}
where for all $b \in \Z$ we have
\begin{align*}
\LLogDR_{A,b}^\gl &= \LLogDR(A,b,-b) \cdot \Phi(\max(\delta_{A,b,-b}^\theta,0)) \\ 
&= \LLogDR(A,b,-b) \cdot \Phi(\max(-\delta_{A,b,-b}^\theta,0))
\end{align*}
and we have $\LLogDR_{A,b} = 0$ if $|b| \geq \sum_{i : a_i > 0} a_i + |k|(2g-2+n)$.
\end{intheorem}

For any piecewise linear function $f$ on $\Mpt_{g,n}^\trop$, the class $\Phi(f)$ is a linear combination of boundary divisors and $\psi$-classes by \cite[Lemma~4.3]{Holmes2023LogarithmicCohomologicalFT}. In \ref{sec:pullbackDR} we show the pushforward of $\LLogDR_{A,b}^\gl$ to $\CH(\Mbar_{g,n})$ can be explicitly computed in terms of smaller double ramification cycles by writing $\Phi(\max(\delta_{A,b,-b}^\theta,0))$ as a sum of boundary divisors for $b \geq 0$, and similarly for $\Phi(\max(-\delta_{A,b,-b}^\theta,0))$ for $b \leq 0$. This leads to the following formula, for $k = 0$ first proven by Zvonkine in \cite{zvonkine}.

\begin{intheorem}[\ref{def:banana},\ref{thm:secondformula}]
\label{inthm:secondformula}
Write $S_{A,b}$ for the set of pairs
\[
  \{(\Gamma,B)\}
\]
where $\Gamma$ is a banana graph\footnote{A stable graph with two vertices $v_1,v_2$, and no loops, with leg $p_{n-1}$ incident to $v_1$ and $p_n$ incident to $v_2$}, and $B \in \Z^{E(\Gamma)}$ with sum $-b - \sum_{i \in L(v_1)} a_i + k(2g(v_1) - 2 + n(v_1))$. If $b \geq 0$, we require all $b_e > 0$ and if $b < 0$ we require $b_e < 0$.

For $(\Gamma,B) \in S_{A,b}$, we write $C_1$ for $((a_i)_{i \in L(v_1)},b,(b_e)_{e \in E(\Gamma)}) \in \Z^{H(v_1)}$ and 
$C_2$ for $((a_i)_{i \in L(v_2)},-b,(-b_e)_{e \in E(\Gamma)}) \in \Z^{H(v_2)}$.  
We have
\begin{equation}
\begin{split}
\gl^* \DR(A) = \sum_{b} \sum_{(\Gamma,B) \in S_{A,b}} \frac{\prod_{e \in E(\Gamma)}|b_e|}{\Aut(\Gamma,B)} \zeta_{\Gamma,*}\left(\DR(C_1)\tensor\DR(C_2)\right).
\end{split}
\end{equation}
\end{intheorem}

The equalities in \ref{inthm:firstformula} also give rise to a special case of the following tautological relations. We use the language of pierced curves to give a short geometric proof of these relations, first proven in \cite{costantini2025integralspsiclassestwisteddouble} using relations to strata of differentials for a dense set of $A$, and polynomiality of the double ramification cycle \cite{Pixton2023DRPoly,spelier2024polynomiality}.

\begin{inproposition}[{\cite[Proposition~3.1]{costantini2025integralspsiclassestwisteddouble}}, \ref{prop:splitpsi}]
\label{inprop:splitpsi}
Fix $A \in \Z^n$ with sum $k(2g-2+n)$. Then we have
\[
(a_{n-1}\psi_{n-1} - a_n \psi_n) \DR(A) + \sum_{(\Gamma,B) \in S}  \frac{s(\Gamma)}{\Aut(\Gamma,B)} \prod_{e \in E(\Gamma)} |b_e| [\Gamma,\DR(C_1) \tensor \DR(C_2)] = 0
\]
where the sum is over the set $S_{a,b} = \{(\Gamma,B)\}$ where
\begin{itemize}
  \item $\Gamma$ is a banana graph with vertices $v_1,v_2$;
  \item $B \in \Z_{> 0}^E$ with sum $- \sum_{i \in L(v_1)} a_i + k(2g(v_1) - 2 + n(v_1))$;
\end{itemize}
and \[s(\Gamma) = \begin{cases}  -1 &\text{ if } p_{n-1} \in L(v_1) \wedge p_n \in L(v_2) \\ 1 &\text{ if } p_{n-1} \in L(v_2) \wedge p_n \in L(v_1) \\ 0 &\text{ else}\end{cases}\]
and we write $C_1$ for $((a_i)_{i \in L(v_1)},(b_e)_{e \in E(\Gamma)}) \in \Z^{H(v_1)}$ and $C_2$ for $((a_i)_{i \in L(v_1)},(-b_e)_{e \in E(\Gamma)}) \in \Z^{H(v_2)}$.
\end{inproposition}

\subsection{Overview of the paper}
In \ref{sec:background} we recall some basics on log schemes, log stacks, log intersection theory and log Jacobians. In \ref{subsec:back:pierced} we summarise the necessary background from \cite{Holmes2023LogarithmicCohomologicalFT}. In \ref{sec:pierceddr} we define the pierced double ramification cycle and show some simple properties. In \ref{sec:pullbacklogDR} we prove the logarithmic pullback formula \ref{inthm:firstformula}, and in \ref{sec:pullbackDR} we use this logarithmic formula to obtain \ref{inthm:secondformula} and \ref{inprop:splitpsi}.

\subsection{Acknowledgements}
I would like to thank Dan Abramovich, Leo Herr, Thibault Poiret, Georgios Politopoulos and Johannes Schmitt for useful discussions. I am very grateful to David Holmes for his many helpful comments. I thank Georgios Politopoulos and Martin M\"oller for encouraging to include the twisted double ramification cycle. And last, the results of Dimitri Zvonkine \cite{zvonkine} formed the main inspiration for this paper, providing useful insights in the classical double ramification cycle and how gluing maps and tangencies interact.

The author is supported by NWO grant VI.Vidi.193.006 and ERC Consolidator Grant FourSurf 101087365.

\section{Background}
\label{sec:background}

\subsection{Log schemes}
In this paper, we mostly work in the category $\LogSch = \LogSch^\fs$ of fs (fine and saturated) log schemes over a point with the trivial log structure.

A log scheme is fs if the log structure is locally pulled back from a chart $\Z[P]$ where $P$ is a finitely generated, integral, saturated monoid. Pierced curves (see \ref{subsec:back:pierced}) are \emph{not} fs.

\begin{definition}
We say a log scheme $X$ is quasi-fs, or qs, if it has local charts $\Z[P]$ where $P$ is integral, saturated and $P^\gp$ is finitely generated.

We let $\LogSch^\qs$ denote the category of qs log schemes.
\end{definition}

In \ref{subsec:back:pierced} we will see pierced curves are qs. Qs monoids naturally appear when extending sharp maps of monoids.
\begin{example}
Let $M$ be a sharp monoid, so $M = \ol{M}$. Take $Q = M \oplus \N$, and consider a sharp section $p: Q \to M$ sending $(0,1)$ to $\ell \in M$ with $\ell \ne 0$. Then there is a unique maximal monoid $P$ with $Q \subset P \subset Q^\gp$ such that $p$ extends to a sharp morphism $P \to M$.

We have \[P = \{(m,a) \in M \oplus \Z : (m,a) = 0 \text{ or } m + a\ell > 0\}\]

This monoid is qs, but not fs.
\end{example}

All fiber products/fiber squares in this paper are in $\LogSch^\fs$.

\subsection{Log stacks}
\begin{definition}
An \emph{algebraic log stack (or space)} is an algebraic stack (or space) equipped with a log structure. 
\end{definition}

\begin{definition}
A \emph{log stack} is a stack fibred in groupoids over $\LogSch$ for the strict \'etale topology that admits a log \'etale cover by a log algebraic stack, such that the diagonal is representable by log algebraic spaces.
\end{definition}

\begin{definition}
A log stack $Y$ is \emph{dominable} if there is a log blowup $X \to Y$ such that $X$ is a log algebraic stack.
An algebraic log stack $X$ is locally free if every stalk of $\Mbar_X$ is isomorphic to $\N^r$ for some $r$.
\end{definition}

\begin{example}
\label{ex:glog}
Any algebraic log stack is a log stack. The following non-algebraic log stacks play an important role in this paper.
\begin{itemize}
\item The log stacks $\G_\log$ and $\G_\trop$, defined by
\[
  \G_\log: X \mapsto \M_X(X), \G_\trop: X \mapsto \ghost_X(X).
\]
The subfunctors
\[
X \mapsto \{m \in \M_X(X) : m \leq 1 \text{ or } m \geq 1 \}, X \mapsto \{m \in \ghost_X(X) : m \leq 0 \text{ or } m \geq 0 \} 
\]

are representable by $\P^1$ and $[\P^1/\G_m]$ respectively, where $\P^1$ has its toric log structure. The corresponding maps $\P^1 \to \G_\log, [\P^1/\G_m] \to \G_\trop$ are log blowups.
\item The logarithmic Jacobian $\LogJ(C/S)/S$ of a log curve $C/S$ \cite{Molcho2018The-logarithmic}. This parametrises \emph{log line bundles} on $C/S$. See \ref{subsec:back:jac} for more on this stack.
\end{itemize}
\end{example}

\subsection{Logarithmic Chow}
\label{sub:logarithmic_chow}

The logarithmic Chow ring of a smooth log smooth algebraic log DM stack was first introduced in \cite{Holmes2017Multiplicativit}, and further developed in \cite{Holmes2021Logarithmic-int,Holmes2022Logarithmic-double}. For $X$ a smooth log smooth DM stack they define 
\begin{equation}
\label{eq:defhps}
\LogCH_{\text{HPS}}^*(X) = \colim_{\tilde{X} \to X} \CH^*(\tilde{X}). 
\end{equation}
where the colimit is taken over all smooth log blowups, and $\CH^*$ is Vistoli's Chow ring with $\Q$-coefficients as in \cite{Vistoli1989Intersection-th}. The transition maps are simply the pullback maps $\pi^*: \CH(\tilde{X} \to \hat{X})$ for a blowup $\tilde{X} \to \hat{X}$.

The stack $\Mpt_{g,n}$ is not log smooth, only idealised log smooth. It has log blowups of higher dimension, and not every blowup is dominated by a smooth blowup. Beyond log smooth schemes, there is both a notion of homological $\LogCH_*(X)$ and bivariant $\LogCH^*(X)$, due to \cite{Barrott2019Logarithmic-Cho}.
\begin{definition}[{\cite{Barrott2019Logarithmic-Cho}}]
Let $X$ be a dominable log stack. Then its homological log Chow group is
\[
\LogCH_*(X) = \colim_{\tilde{X} \to X} \CH_*(\tilde{X})
\]
where the transition maps along $\pi: \tilde{X} \to \hat{X}$ are the virtual dimension $0$ Gysin pullback maps $\pi^!$ constructed in \cite[Construction~2.4]{Barrott2019Logarithmic-Cho}, and $\CH_*$ denotes Kresch's Chow groups \cite{Kresch1999Cycle-groups-fo}.

Its cohomological log Chow ring
\[
\LogCH^*(X)
\]
is the ring of bivariant classes on $X$ \cite[Definition~2.20]{Barrott2019Logarithmic-Cho}.
\end{definition}

For a more in depth discussion on the homological Chow ring, see \cite{pandharipande2024logtautrings}. For a comparison between the bivariant classes and \ref{eq:defhps}, see \cite[Appendix A]{Holmes2023LogarithmicCohomologicalFT}.


\subsection{Pierced curves}
\label{subsec:back:pierced}

Pierced curves were first defined in \cite[Section~5]{Holmes2023LogarithmicCohomologicalFT}. In this section we summarise the definitions and results needed for this paper.

We first remind the reader of the definition of a log curve
\begin{definition}[\cite{Kato2000Log-smooth-defo}]
A \emph{log curve} is a morphism of log algebraic spaces $C \to S$ that is proper, integral, saturated, log smooth, with geometric fibers that are reduced, connected and of pure dimension $1$.
\end{definition}

A log curve $C/S$ has a locus of non-verticality, where $\ghost_{C,p} = \ghost_S \oplus \N$. Locally on $S$ this is a disjoint union of schematic sections, but these schematic sections do not in general lift to logarithmic sections.

\begin{definition}[\protect{\cite[Definition~3.1]{Holmes2023LogarithmicCohomologicalFT}}]
A \emph{log pointed curve} is a tuple $(C/S,p_1,\dots,p_n)$ where
\begin{enumerate}
  \item $C/S$ is a log curve;
  \item the $p_i$ are sections $S \to C$ of log schemes, landing in the locus where $C/S$ is classically smooth;
  \item $C/S$ is vertical exactly outside the image of the sections.
\end{enumerate}

We define the \emph{length} $\ell_i \in \ghost_S$ of $p_i$ to be the image of $(0,1) \in \ghost_{C,p_{i}} = \ghost_S \oplus \N$ under $p_i^*: \ghost_{C,p_i} \to \ghost_S$.

The moduli space of log $n$-pointed curves of genus $g$ is denoted $\Mpt_{g,n}$
\end{definition}

Per \cite[Lemma~3.7]{Holmes2022Logarithmic-double} we have that $\ell$ maps to $0 \in \Ocal_S/\Ocal_S^\times$ under the monoid morphism $\ghost_S \to \Ocal_S/\Ocal_S^\times$.

\begin{theorem}[\protect{\cite[Theorem~A]{Holmes2023LogarithmicCohomologicalFT}}]
The stack $\Mpt_{g,n}$ is an algebraic stack with logarithmic structure. The forgetful map $\Mpt_{g,n} \to \Mbar_{g,n}$ is an isomorphism of underlying algebraic stacks. There are natural gluing maps
\begin{align*}
\Mpt_{g_1,n_1 + 1} \times \Mpt_{g_2,n_2+1} &\to \Mpt_{g_1+g_2,n_1+n_2}\\
\Mpt_{g,n} &\to \Mpt_{g,n-1}
\end{align*}
compatible with the classical gluing maps.
\end{theorem}

Tropically, a pointed curve is roughly a metric graph where the leg is also given a length. Gluing two legs then corresponds to creating an edge with length the sum of the two leg lengths. This tropical perspective is worked out in \cite[Appendix~B]{Holmes2023LogarithmicCohomologicalFT}.

\begin{definition}[{\cite[Definition~5.1]{Holmes2023LogarithmicCohomologicalFT}\footnote{Condition \oref{cond:changed} is slightly updated from loc. cit. This will also be updated in the next version of loc. cit. The gluing result \cite[Theorem~5.10]{Holmes2023LogarithmicCohomologicalFT} holds for either definition}}]
Let $C/S$ be a pointed curve. The \emph{piercing} of $C/S$ at $(p_1,\dots,p_n)$ is the unique log scheme $\pie{C}/C$ such that
\begin{enumerate}
  \item $\ul{\pie{C}} \to \ul{C}$ is an isomorphism;
  \item $\pie{C} \to C$ is an isomorphism on $C \setminus \{p_1,\dots,p_n\}$;
  \item \label{cond:changed} at $p_i$, we have \[ \M_{C,p_i} = \M_S \oplus \N \subset \M_{\pie{C},p_i} = \{(m,a) : a \geq 0 \text{ or } \alpha(p_i^*(m,a)) = 0\} \subset \M_S \oplus \Z.\]
\end{enumerate}
\end{definition}

The gluing of pointed curves can be stated in terms of a coproduct of pierced curves. This gives rise to the following universal property that plays a key role in this paper.

\begin{theorem}[{\cite[Theorem~5.10]{Holmes2023LogarithmicCohomologicalFT}}]
\label{thm:gluing}
Let $C$ be a possibly disconnected pointed curve with two markings $p_1, p_2$. Let $C^\gl$ be the gluing of $C$ at $p_1,p_2$. Let $X$ be a log stack. There is a natural map $\pie{C} \to C^\gl$ that induces a bijection
\begin{align*}
\Hom(C^\gl, X) \to \Hom_{\gl}(\pie{C}, X) &= \Bigg\{f \in \Hom(\pie{C},X) \hspace{4pt}\Bigg\vert \begin{split}& f \circ p_1 = f \circ p_2 \\ &\text{slopes of } f \text{ at } p_1,p_2 \text{ add to } 0\end{split}\Bigg\} \\
\end{align*}
\end{theorem}


\begin{remark}
\label{rem:piercedvspunctured}
Given a pierced curve, one can associate a punctured curve, see \cite[Section~3.2]{Holmes2023LogarithmicCohomologicalFT}. Punctured and pierced curves both admit stable maps with negative tangencies, but punctured curves or punctured stable maps cannot be glued.
\end{remark}

\subsection{Logarithmic Jacobians and stability conditions}
\label{subsec:back:jac}

Let $C/S$ be a log curve with no markings. Then its log Jacobian $\LogJ(C)/S$ \cite{Molcho2018The-logarithmic} is a logarithmic abelian variety and in particular is proper. It is a compactification of the semi-abelian variety $\Jac^{\ul{0}}(S)$ of multidegree $0$ line bundles, and parametrises \emph{log line bundles}, $\G_\log$-torsors. By extension of constants from $\G_m$ to $\G_\log$, there is a map $\Jac^0(S) \to \LogJ(C)/S$, which is in general neither injective nor surjective.

Recall that the short exact sequence 
\[
  \Ocal_C^* \to \M_{C}^\gp \to \ghost_C^\gp
\]
induces the sequence
\[
\PL(C) \xrightarrow{\alpha \mapsto \Ocal_C(\alpha)} \Jac^0(C) \to \LogJ(C)
\]
exact at $\Jac^0(S)$.

For a log curve with markings, one needs to restrict to piecewise linear functions with slope $0$ on the legs. The piecewise linear function $\alpha$ with slopes $a_i$ at leg $i$, and slopes $0$ at internal edges has line bundle $\Ocal_C(\alpha) = \Ocal_C(\sum_i a_i p_i)$, which is not necessarily trivial as a log line bundle.

All in all, this gives us the following lemma.
\begin{lemma}[{\cite{Molcho2018The-logarithmic}}]
\label{lem:loglinebundlekern}
Let $C/S$ be a log curve. The kernel of $\Jac^0(S) \to \LogJ(C)/S$ is given by line bundles of the form $\Ocal_C(\alpha)$, where $\alpha$ is a piecewise linear function on $C/S$ with slope $0$ at the legs.
\end{lemma}

The logarithmic Jacobian is \emph{not} algebraic for curves that are not of compact type. It has algebraic log blowups, for example those given by small non-degenerate Kass-Pagani stability conditions $\theta$ (see \cite{Kass2017The-stability-s} and \cite[Section~1.6,4]{Holmes2022Logarithmic-double}). These compactified Jacobians parametrise line bundles on log blowups (bubblings) of the original curve with $\theta$-stable multidegree. A line bundle of multidegree $0$, in particular the trivial line bundle, is always $\theta$-stable.

For a fixed stability condition $\theta$, and a line bundle $\Lcal$ on $C/S$, there is a piecewise linear function $\alpha^\theta$ (with slope $0$ on the legs, unique up to constants) such that $\Lcal_A(\alpha^\theta)$ is $\theta$-stable.

A stability condition induces a stability condition on strata. In particular, a stability condition on $\Mbar_{g,n}$ induces a stability condition on $\Mbar_{g-1,n+2}$. 

\section{Pierced logarithmic double ramification cycle}
\label{sec:pierceddr}

For $A \in \Z^n$ with sum $k(2g-2+n)$, the \emph{double ramification locus}
\[
\DRL_g(A) \subset \Mbar_{g,n}
\]
is the locus where the line bundle $\omega_{C}^{\log,\tensor -k}(\sum_i a_i p_i)$ on the curve is trivial. On the smooth locus, this consists of curves $C/S$ with a map $C \to \P^1$, with ramification profile over $0$ and $\infty$ specified by positive and negative $a_i$ respectively. The \emph{double ramification cycle} $\DR_g(A) \in \CH^g(\Mbar_{g,n})$ is the virtual fundamental class of a compactification.

The (logarithmic) double ramification cycle has seen many definitions and incarnations. The common definitions are either using maps to chains of $\P^1$'s \cite{Graber2003Hodge-integrals} (when $k = 0$), or resolutions of the Abel--Jacobi section $\Mbar_{g,n} \dashrightarrow \Jac_{g,n}^{\ul{0}}$ \cite{Holmes2017Extending-the-d}. We present an approach continuing on \cite{Marcus2017Logarithmic-com,Holmes2021Logarithmic-int} mixing the two, using the theory of logarithmic Jacobians and logarithmic intersection theory.

\begin{definition}
Fix $g,n \in \Z_{\geq 0},A \in \Z^n$ with $\sum a_i = k(2g-2+n)$. The log double ramification locus $\LogDRL_g(A)$ is defined as the log fiber product
\[
\begin{tikzcd}
\LogDRL_g(A) \arrow[r] \arrow[d,"j"] & \Mbar_{g,n} \arrow[d,"e"] \\
\Mbar_{g,n} \arrow[r, "\sigma_A"] & \LogJ_{g,n}
\end{tikzcd}
\]
where $e$ is the unit of $\LogJ_{g,n}$ and $\sigma_A$ sends $C/S$ to the log line bundle $\omega_{C}^{\log,\tensor -k}(\sum_i a_i p_i)$.

The log double ramification cycle $\LogDR_g(A) \in \LogCH_*(\Mbar_{g,n})$ is defined as $j_* e^![\Mbar_{g,n}]$.
\end{definition}

Equivalently, we have that $\LogDRL_g(A)$ is the locus where $\omega_{C}^{\log,\tensor -k}(\sum_i a_i p_i)$ is trivial as a \emph{log} line bundle.

\begin{remark}
As this definition is slightly different from literature, we show how to unpack it. The space $\LogDRL_g(A)$ is a proper algebraic log stack, and carries the virtual fundamental class $e^![\Mbar_{g,n}]$ of dimension $2g-3+n$. The map $\LogDRL_g(A) \to \Mbar_{g,n}$ is not a closed embedding on the underlying algebraic stacks. Instead, the map factors (non-canonically) as a closed embedding in a log blowup $\tilde{\Mcal}_{g,n}/\Mbar_{g,n}$. The pushforward of $[\LogDRL_g(A)]^\vfc$ to $\tilde{\Mcal}_{g,n}$ lives in $\CH^g(\tilde{\Mcal}_{g,n})$. Then the logarithmic double ramification cycle is the class of this pushforward inside $\LogCH^g(\Mbar_{g,n})$. 
\end{remark}

We will also need the pierced version of this story. 

\begin{definition}
\label{def:pierceddr}
Let $\pi: \Mpt_{g,n} \to \Mbar_{g,n}$ be the forgetful map. We define
\begin{align*}
\LLogJ_{g,n}/\Mpt_{g,n} &= \pi^*(\LogJ/\Mbar_{g,n}), \\
\LLogDRL_{g,n}/\Mpt_{g,n} &= \pi^*(\LogDRL/\Mbar_{g,n}), \\
\LLogDR(A) &= \pi^*\LogDR(A) \in \LogCH_{*}(\Mpt_{g,n}).
\end{align*}
\end{definition}

Note that by the residue map we have canonical trivialisations $\res: \omega_{C}^{\log,\tensor -k}|_{p_i} \isom \Ocal_S$.

The following functor of points definition of $\LLogDRL_g(A)$ is shown in \cite[Section~7]{Holmes2023LogarithmicCohomologicalFT}.
\begin{proposition}
\label{lem:funcpointsdr}
The functor of points of $\LLogDRL_g(A)$ is
\[
((C/S,p_1,\dots,p_n), P/C, f: C \to P)
\]
where $(C/S,p_1,\dots,p_n)$ is a log pointed curve, $P/C$ is the $\G_\log$-torsor induced by the line bundle $\omega_{C}^{\log,\tensor -k}$, and $f: C \to P$ is a section where the slope on the leg $p_i$ is $a_i$ and $\res(f(p_1)) = 1 \in \G_\log$.
\end{proposition}

\begin{definition}
\label{def:evmaps}
We write $\ev_i$ for the composition $\res \circ f \circ p_i: \LLogDRL_g(A) \to P|_{p_i} \to \G_\log$.
\end{definition}

\section{The pullback of the log double ramification cycle}
\label{sec:pullbacklogDR}

Fix $g \in \Z_{\geq 0}, n \in \Z_{\geq 2}$. We fix a non-degenerate small stability condition $\theta$ on $\Mpt_{g+1,n-2}$. This also induces a stability condition $\theta$ on $\Mpt_{g,n}$, as $\Mpt_{g,n}$ is a stratum of $\Mpt_{g+1,n-2}$. We fix $A \in \Z^{n-2}$ with sum $k(2(g+1)-2+(n-2)) = k(2g-2+n)$.

To state the formula, we first define the following piecewise linear functions on the universal curve $C/\Mpt_{g,n}$ and $\Mpt_{g,n}$.
\begin{definition}
\label{def:alphabtheta}
For $B \in \Z^n$ with sum $k(2g-2+n)$, let $\alpha_{B}^\theta \in \PL(C)$ be the unique piecewise linear function that satisfies
\begin{enumerate}
\item $\alpha_B^\theta$ has slope $b_i$ along leg $i$;
\item the line bundle $\omega_{C}^{\log,\tensor -k}(\alpha_B^\theta)$ is $\theta$-stable;
\item $\alpha_{B}^\theta(p_1) = 0$ (where $\alpha_B^\theta(p_1) = \alpha_B^\theta(v) + b_i \ell_i$ if leg $i$ is incident to vertex $v$).
\end{enumerate}
Write $\delta_B^\theta$ for $\alpha_B^\theta(p_{n}) - \alpha_B^\theta(p_{n+1})$, a piecewise linear function on $\Mpt_{g,n}$.
\end{definition}

\begin{lemma}
\label{lem:linebundlepl}
For $B \in \Z^n$ with sum $k(2g-2+n)$, the log line bundles corresponding to the two line bundles $\Ocal_C(\sum_i b_i p_i)$ and $\Ocal_C(\alpha_B^\theta)$ are isomorphic.
\end{lemma}
\begin{proof}
Recall that by \ref{lem:loglinebundlekern} a log line bundle is trivial if it is of the form $\Ocal_C(\beta)$ for $\beta$ a PL function with slopes $0$ on the legs. Write $\beta_B^\theta$ for the PL function on $C$ that agrees with $\alpha_B^\theta$ everywhere except the legs, and has slope $0$ on the legs. Then by construction $\Ocal_C(\alpha_B^\theta) = \Ocal_C(\sum_i b_i p_i) \tensor \Ocal_C(\beta_B^\theta)$.
\end{proof}

\begin{remark}
Note that $\alpha_B^\theta$ and $\delta_B^\theta$ are not strict piecewise linear. They are strict piecewise linear on (the universal curve over) $\LLogDRL_g(B)$. In \cite{Holmes2022Logarithmic-double} they give subdivisions of $\Mpt_{g,n}$ and the universal curve that support respectively $\alpha_A^\theta$ and $\delta_A^\theta$.
\end{remark}

We give some examples that exhibit the behaviour of $\delta_A^\theta$.

\begin{example}
We take $g = 0, n = 4$ and $A = (a_1,a_2,a_3,a_4)$. As we are working in genus $0$, a line bundle is $\theta$-stable if and only if it has multidegree $0$. \ref{fig:deltaalpha04} contains a picture of the four strata of $\Mbar_{0,4}$, corresponding to the origin and the three rays. Also pictured are the corresponding tropical curves, the slopes of the piecewise linear function $\alpha_A$, and the value of $\delta_A$. We see the value of $\delta_A$  depends on whether legs $3$ and $4$ are incident to the same vertex. If they are, we have $\delta_A^\theta = a_4 \ell_i - a_3 \ell_i$, and else we have $\delta_A^\theta = a_3 \ell_i - a_4 \ell_i + a\ell$, where $\ell$ is the length of the unique inner edge of the graph and $a$ is the slope along this edge (oriented from leg $3$ to leg $4$), determined by the multidegree $0$ condition.

\tikzmath{\r1 = 4; \lw =.6; \dx = .6; \dy = .9; \el = 1.8;}

\definecolor{mycolor}{RGB}{230,0,0}

\newcommand{\DrawGraph}[8]{%
    \begin{scope}[shift={(#5,#6)}]
        \begin{scope}[every node/.style={circle, draw,fill=black!50,inner sep=0pt, minimum width=4pt}] 
        \node (A) at (0,0) {};
        \node (B) at (\el,0) {};
        \end{scope}
        \node (A1) at (-\dx,\dy) [left] {#1};
        \node (A2) at (-\dx,-\dy) [left] {#2};
        \node (B1) at (\el+\dx,\dy) [right] {#3};
        \node (B2) at (\el+\dx,-\dy) [right] {#4};

        \draw[->] (A) -- node [pos=0.5, above, blue] {$#7$}(B);
        \draw[->] (A) -- node [pos=0.7, right, blue] {$a_{#1}$} (-\dx,\dy);
        \draw[->] (A) -- node [pos=0.7, right, blue] {$a_{#2}$} (-\dx,-\dy);
        \draw[->] (B) -- node [pos=0.7, left, blue] {$a_{#3}$} (\el+\dx,\dy);
        \draw[->] (B) -- node [pos=0.7, left, blue] {$a_{#4}$} (\el+\dx,-\dy);

        \node at (\el + \dx + .5, 0) [right, mycolor, text width = 2.7cm] {$\delta_{A} = #8$};
    \end{scope}
}

\begin{figure}
\centering

    \begin{tikzpicture}[scale=.9]

        \draw[->, line width = \lw] (0,0) -- (90:\r1);
        \draw[->, line width = \lw] (0,0) -- (210:\r1);
        \draw[->, line width = \lw] (0,0) -- (330:\r1);

        \fill (90:\r1*.66) circle (1.5pt);
        \draw[dotted, line width = \lw] (90:\r1*.66) -- (2,3);

        \DrawGraph{1}{2}{3}{4}{3}{3}{a_3+a_4}{a_4 \ell_i - a_3 \ell_i};

        \fill (210:\r1*.8) circle (1.5pt);
        \draw[dotted, line width = \lw] (210:\r1*.8) -- (-5.2,-.4);

        \DrawGraph{1}{3}{2}{4}{-6.3}{.8}{a_2+a_4}{a_4 \ell_i - a_3 \ell_i + (a_2 + a_4) \ell};

        \fill (330:\r1*.8) circle (1.5pt);
        \draw[dotted, line width = \lw] (330:\r1*.8) -- (3.3,-1);

        \DrawGraph{1}{4}{2}{3}{3}{0}{a_2+a_3}{a_4 \ell_i - a_3 \ell_i  - (a_2 + a_3) \ell};

    \end{tikzpicture}
  \label{fig:deltaalpha04}
  \caption{The piecewise linear function $\delta_A$ on $\Mpt_{0,4}^\trop = \Mbar_{0,4}^\trop \times \R_{\geq 0}^4$. The cone complex in the middle is $\Mbar_{0,4}^\trop$. For each maximal stratum, a corresponding pierced tropical curve is shown. In blue are the slopes of $\alpha_A^\theta$, and in red the value of $\delta_A$ on this pierced tropical curve.}
\end{figure}
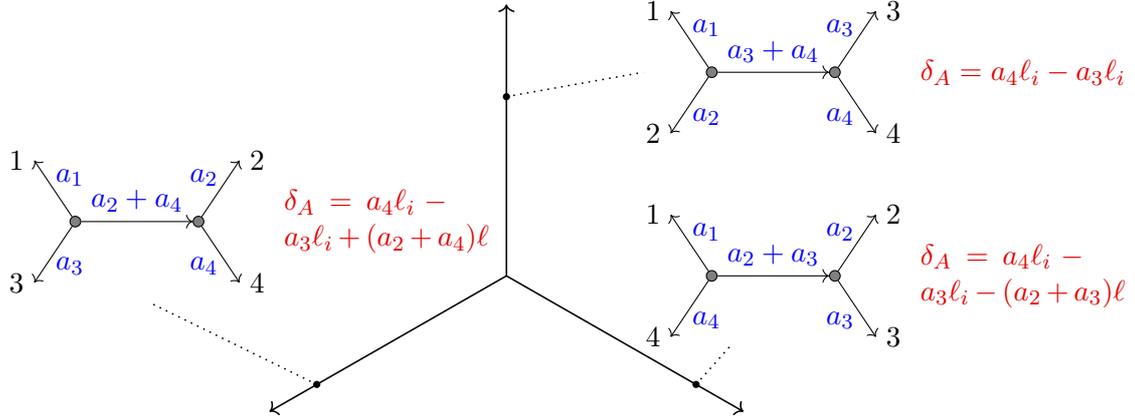

We see that $\delta_A^\theta$ depends on both the leg lengths and the lengths of the inner edges.
\end{example}

\begin{example}
We take $g = 1, n = 2$ and $A = (a,-a)$. Let $\Gamma$ be the graph with two vertices ${v_1,v_2}$, two edges $e_1,e_2$ connecting them, and with one leg $p_i$ each. If $a \geq 2$, then on (the universal curve over) $\Mpt_{\Gamma}$ the piecewise linear functions $\alpha_A^\theta,\delta_A^\theta$ are not strict and depend on $\theta$. See \cite[Section~4.3]{Holmes2022Logarithmic-double} for a genus $2$ example.
\end{example}

In this section we will prove the following formula.
\begin{theorem}
\label{thm:firstformula}
Let $\gl$ be the map $\Mpt_{g,n} \to \Mpt_{g+1,n-2}$. Let $A \in \Z^{n-2}$with sum $k(2g-2+n)$. We have a splitting 
\begin{equation}
\label{eq:formulaevtrop}
\gl^* \LLogDR(A) = \sum_{b \in \Z} \LLogDR_{A,b}^\gl
\end{equation}
where for all $b \in \Z$ we have
\begin{align*}
\LLogDR_{A,b}^\gl &= \LLogDR(A,b,-b) \cdot \Phi\left(\frac12|\delta_{A,b,-b}^\theta|\right) \\
&= \LLogDR(A,b,-b) \cdot \Phi\left(\max(\delta_{A,b,-b}^\theta,0)\right) \\
&= \LLogDR(A,b,-b) \cdot \Phi\left(\max(-\delta_{A,b,-b}^\theta,0)\right).
\end{align*}
and we have $\LLogDR_{A,b} = 0$ if $|b| \geq \sum_{i : a_i > 0} |a_i| + |k|(2g-2+n)$.
\end{theorem}

\begin{remark}
We remark that $\Phi(\max(\delta_{A,b,-b}^\theta,0))$ and $\Phi(\max(-\delta_{A,b,-b}^\theta,0))$ are \emph{not} equal. In \ref{prop:splitpsi} we will show the relation $\LLogDR(A,b,-b) \cdot \delta_{A,b,-b}^\theta = 0$ obtained from \ref{thm:firstformula} gives rise to a tautological relation between double ramification cycles, first proven in \cite[Proposition~3.1]{costantini2025integralspsiclassestwisteddouble}.
\end{remark}

In \ref{subsec:split} we show that the pullback $\gl^* \LLogDRL(A)$ of the log double ramification locus splits into finitely many connected components, indexed by a slope $b$. In \ref{prop:glogpullback} we will show that each component has a closed embedding in $\LLogDRL(A,b,-b)$, and that this embedding is a pullback of the codimension $1$ regular embedding $1 \to \G_\log$. In \ref{subs:compvfc} we will show the Gysin pullback of $\LLogDR(A,b,-b)$ along this embedding gives the correct virtual fundamental class on the component of $\gl^* \LLogDRL(A)$. Finally, in \ref{subsec:proof} we will compute this Gysin pullback, and show it equals $\LLogDR(A,b,-b) \cdot \Phi(\frac12|\delta_{A,b,-b}^\theta|)$.

As $\gl^* \LLogDR(A)$ does not depend on the stability condition, we also obtain relations if we take differing stability conditions.

\begin{corollary}
Let $\theta,\theta'$ be two different choices of stability conditions. Then we have
\[
\LLogDR(A,b,-b) \cdot \Phi(|\delta_{A,b,-b}^\theta| - |\delta_{A,b,-b}^{\theta'}|) = 0.
\]
\end{corollary}

In the special case $A = (0,\dots,0)$ we have $\LLogDR(A) = \DR(A) = \lambda_{g+1}$, and we obtain the following vanishing, shown in \cite{faber00logseries} of the pullback of $\lambda_{g+1}$ as a corollary.

\begin{corollary}
We have the equality \[\gl^* \lambda_{g+1} = 0.\]
\end{corollary}

We also immediately obtain the following vanishing.

\begin{corollary}
We have \[\gl^* \LogDR(A)^2 = 0\].
\end{corollary}
\begin{proof}
As $\LLogDRL_{A,b}^\gl, \LLogDRL_{A,b'}^\gl$ are disjoint for $b \ne b'$, we obtain
\begin{align*}
\gl^* \LogDR(A)^2 = \sum_b \LLogDR(A,b,-b)^2 \cdot \Phi(\max(\delta_{A,b,-b}^\theta,0)) \cdot \Phi(-\max(\delta_{A,b,-b}^\theta,0))
\end{align*}
hese last two factors are supported on the disjoint subcycles where $\delta_{a,b,-b}^\theta > 0$ and $\delta_{a,b,-b}^\theta < 0$ respectively, so each term vanishes.
\end{proof}

\subsection{Splitting of the pullback of the logarithmic double ramification cycle}
\label{subsec:split}

We first set up some notation.

\begin{definition}
We define $\LLogDRL_A^\gl$ by the pullback square
\[
\begin{tikzcd}
\LLogDRL_A^\gl \arrow[r] \arrow[d] \pb & \arrow[d] \LLogDRL(A)\\ 
\Mpt_{g,n} \arrow[r] & \Mpt_{g+1,n-2}
\end{tikzcd}
\]
and similarly define $\LLogJ_{g,n}^{\gl}$ by the pullback square
\[
\begin{tikzcd}
\LLogJ_{g,n}^{\gl} \arrow[r] \arrow[d] \pb & \LLogJ_{g+1,n-2} \arrow[d] \\ 
\Mpt_{g,n} \arrow[r] & \Mpt_{g+1,n-2}
\end{tikzcd}
\]
\end{definition}


\begin{definition}
\label{def:complog}
Let $s: \LLogDRL_{A}^\gl \to \Z$ denote the map measuring the slope of $\alpha_A^\theta$ along the glued edge, oriented from $p_{n-1}$ to $p_{n}$.
Define \[\LLogDRL_{A,b}^\gl \coloneqq s^{-1}(b).\]
\end{definition}

The function $s$ is locally constant and we obtain a splitting
\[
\LLogDRL_{A}^\gl = \bigsqcup_{b\in\Z} \LLogDRL_{A,b}^\gl
\]

\begin{remark}
In fact, one can obtain $s$ from a map of algebraic stacks $\LLogDRL_{A}^\gl \to \sqcup_{b \in \Z} \B \G_m$ using the theory of \emph{cone stacks with boundary} developed in \cite[Section~3]{pandharipande2024logtautrings}. We construct this map in this remark.

By \cite[Remark~42]{pandharipande2024logtautrings} the category of cone stacks with boundary is equivalent to reduced closed substacks of Artin fans. Let $\Gamma$ be the graph with one vertex of genus $g$, and a single, oriented loop $e$. This gives a stratum $\Mbar_{\Gamma}^\trop$ of $\Mbar_{g+1,n-2}^\trop$. We obtain a tropicalisation map $\LLogDRL_A^\gl \to \Mbar_{\Gamma}^\trop$. 

Let $\Sigma$ be the cone stack with boundary consisting of piecewise linear functions on $e$ that vanish at the start of $e$. Explicitly, this has $\R$-points given by pairs $(\ell, f: [0,\ell] \to \R)$ with $\ell \in \R_{> 0}$, with $f$ piecewise linear with integer slopes and $f(0) = 0$. This contains the cone stack with boundary $\Sigma_1$ of linear functions on $e$ that vanish at the start of $e$. This is isomorphic to $\Z \times \R_{>0}$, and corresponds to the Artin stack{}
\[
\sqcup_{b \in \Z} \B \G_m.
\]

Then $\alpha_A^\theta$ is a piecewise linear function on the universal tropical curve over $\Mbar_{\Gamma}^\trop$, and there is a map of cone stacks with boundary $\Mbar_{\Gamma}^\trop \to \Sigma$. The function $\alpha_A^\theta$ is strict piecewise linear over the double ramification locus, so the composition $\LLogDRL_A^\gl \to \Mbar_{\Gamma}^\trop$ factors through $\Sigma_1$. We obtain a map of algebraic stacks $\LLogDRL_{A}^\gl \to \sqcup_{b \in \Z} \B \G_m$.

Now we see $s$ is a continuous map to a discrete set, and in particular locally constant.
\end{remark}

\begin{definition}
Let $[\LLogDRL_{A,b}^\gl]^\vfc$ denote the virtual fundamental class $\gl^! \LLogDRL_{A}^\gl$ on $\LLogDRL_{A,b}^\gl$. Let $\LLogDR_{A,b}^\gl$ be the pushforward of $[\LLogDRL_{A,b}^\gl]^\vfc$ to $\LogCH(\Mpt_{g,n})$. 
\end{definition}

We obtain the splitting
\[
\gl^! \LLogDR(A) = \sum_{b \in \Z} \LLogDR_{A,b}^\gl
\]
as desired. This infinite sum has only finitely many non-zero terms, as $\gl^* \LLogDRL(A)$ is proper. Using balancing equations, we can give an explicit bound.


\begin{lemma}
\label{lem:bounded}
The space $\LLogDRL_{A,b}^\gl$ is empty if $b \ne 0$ and $|b| \geq \sum_{i : a_i > 0} a_i + |k| (2g-2+n)$.
\end{lemma}
\begin{proof}
Without loss of generality, assume $b > 0$. 

Let $C/s$ be a geometric point of $\LLogDRL_{A,b}^\gl$. Let $\Gamma = (V,E)$ be the corresponding tropical (unglued) curve. Let $\alpha = \alpha_{A}$ be the piecewise linear function on $\Gamma^\gl$ defined in \ref{def:alphabtheta}. As $C$ lies in the logarithmic double ramification locus, this is a strict piecewise linear function that does not depend on $\theta$, and $\alpha$ satisfies the balancing equation, meaning that for every vertex the sum of outgoing slopes is $0$. 

Write $v_i$ for the vertex that leg $p_i$ is attached to. Denote $V_0 = \{v \in V : \alpha(v) \leq \alpha(v_{n-1})\}$. As $\alpha$ has slope $b$ along the glued edge, we have $\alpha(v_n) = \alpha(v_{n-1}) + b(\ell_{n-1} + \ell_n) > \alpha(v_{n-1})$, so $v_n \not\in V_0$.

The outgoing slopes of $\alpha$ at a vertex $v$ is $-k(2g(v) - 2 + n(v))$. Summing this condition over $V_0$ gives us the equality
\[
b + \sum_{e: v \to w | v \in V_0, w \in V \setminus V_0} s(e) + \sum_{i : a_i > 0, v_i \in V_0} a_i - \sum_{v \in V_0} k(2g(v) - 2 + n(v))  = -\sum_{i : a_i < 0, v_i \in V_0} a_i
\]
where all terms except possibly $k(2g(v) - 2 + n(v))$ are positive. As $\Gamma$ is connected, the total contribution from the edges between $V_0$ and $V \setminus V_0$ is positive.

In particular, we must have $0 < b < - \sum_{i : a_i < 0} a_i = \sum_{i : a_i > 0} a_i + |k|(2g-2+n)$.
\end{proof}

Now we fix a choice of $b \in \Z$, and study a single component $\LLogDRL_{A,b}^\gl$.


\begin{lemma}
\label{lem:logdrsplits}
We have an inclusion\footnote{On the level of underlying algebraic stacks, this is not a monomorphism. On the level of log stacks, it is a log closed embedding (a closed embedding in a log blowup). In particular, it is a log monomorphism. As our log stacks are CFGs over $\LogSch$, it is thence an inclusion of log stacks.}
\[
\LLogDRL_{A,b}^\gl \subset \LLogDRL({A,b,-b})
\]
\end{lemma}
\begin{proof}
Recall that by \ref{lem:funcpointsdr} we have
\[
\LLogDRL({A,b,-b}) = \{((C/S,p_1,\dots,p_n) \in \Mpt_{g,n}, P/C, f: C \to P)\}
\]
where $\alpha$ is required to have slopes $(A,b,-b)$ at the legs, and similarly
\[
\LLogDRL_{A,b}^\gl = \{((C/S,p_1,\dots,p_n) \in \Mpt_{g,n}, P/C, f: C^\gl \to P)\}
\]
where $f$ is required to have slopes $A$ at the legs, and $b$ at the glued edge.

For $((C/S,p_1,\dots,p_n) \in \Mpt_{g,n}, P/S, f: C^\gl \to P) \in \LLogDRL_{A,b}^\gl$, we see the composite $f \circ (C \to C^\gl)$ is a function $f': C \to P$ with the required slope.
\end{proof}

\begin{definition}
\label{def:ev}
Let
\[
\delta_{A,b,-b}^\log: \LLogDRL({A,b,-b}) \to \G_\log
\]
be the map defined as follows.

Let $((C/S, p_1 ,\dots, p_n ), P/S, f : C \to P) \in \LLogDRL({A,b,-b})$. Recall from \ref{def:evmaps} the evaluation map $\ev_{i}$, the composition of the section $p_i: S \to C$, the map $f: C \to P$ and the residue map $\res: P|_{p_i} \to \G_\log$. Then
\[
\delta_{A,b,-b}^\log(C/S) = (-1)^k \ev_{n}/\ev_{n-1} \in \G_\log(S).
\]
\end{definition}
\begin{remark}
The sign $(-1)^k$ appears because the two residue maps $(\omega_C^{\log,\tensor -k})|_q \to \Ocal_S$ at a node $q$ differ by a factor $(-1)^k$.
\end{remark}

\begin{lemma}
The composition of $\delta_{A,b,-b}^\log$ and the quotient map $\G_\log \to \G_\trop$ is equal to the piecewise linear function \[\delta_{A,b,-b}^\theta|_{\LLogDRL({A,b,-b})}: \LLogDRL({A,b,-b}) \to \G_\trop.\]
\end{lemma}
\begin{proof}
Let $((C/S, p_1 ,\dots, p_n ), P/S, f : C \to P) \in \LLogDRL({A,b,-b})$. Note that $P/\G_m$ is a trivialised $\G_\trop$-torsor, and with respect to this map $P \to \G_\trop$ sthe tropicalisation of $f$ is $\alpha_{A,b,-b}$, and hence the tropicalisation of $\delta_{A,b,-b}^\log$ is the piecewise linear function $\delta_{A,b,-b}^\theta$.
\end{proof}

\begin{proposition}
\label{prop:glogpullback}
The square
\[
\begin{tikzcd}
\LLogDRL_{A,b}^\gl \arrow[r] \arrow[d] \pb & \LLogDRL({A,b,-b}) \arrow[d, "\delta_{A,b,-b}^\log"]\\ 
* \arrow[r, "j"] & \G_\log
\end{tikzcd}
\]
is a fiber square, where $j$ is the unit of $\G_\log$.
\end{proposition}
\begin{proof}
A curve $((C/S,p_1,\dots,p_n), P/C, f: C \to P)$ lies in $\LLogDRL_{A,b}^\gl$ if and only if $f$ lifts to a map $f^\gl: C^\gl \to P^\gl$, where $P^\gl$ is the $\G_\log$ torsor on $C^\gl$ coming from $\omega_{C^\gl}^{\log,\tensor - k}$.

By the universal property of gluing \ref{thm:gluing} for the map $C \to P^\gl$, this is equivalent to having opposite slopes and matching evaluation maps at $p_{n-1},p_n$. The opposite slope condition is true by construction, and the evaluations match if and only if $\delta_{A,b,-b}^\log$ takes the value $1$.
\end{proof}

This induces a perfect obstruction theory of virtual relative codimension $1$ on the map $\LLogDRL_{A,b}^\gl \to \LLogDR(A,b,-b)$, given by the Gysin pullback of $1 \to \G_\log$.

\subsection{Comparison of the virtual fundamental class}
\label{subs:compvfc}
Consider the virtual fundamental class $[\LLogDRL_{A,b}^\gl]^\vfc$, given by restriction of $\gl^* \LLogDR(A)$ to the connected component $\LLogDRL_{A,b}^\gl$ defined in \ref{def:complog}. \ref{prop:glogpullback} shows that that $\LLogDRL_{A,b}^\gl \to \LLogDR(A,b,-b)$ has a perfect obstruction theory of virtual relative codimension $1$, coming from the Gysin pullback of the embedding $1 \to \G_\log$. In this subsection we will show a compatibility between this perfect obstruction and the virtual fundamental classes on source and target.

\begin{proposition}
\label{prop:vfceq}
We have an equality of virtual fundamental classes
\[
[\LLogDRL_{A,b}^\gl]^\vfc = \delta_{A,b,-b}^{\log,!} [\LLogDRL({A,b,-b})]^\vfc.
\]
\end{proposition}

To prove this, we will study the behaviour of $\LLogJ_{g,n}^\gl = \gl^*\LLogJ_{g+1,n-2}$. 

For this, one would like to write down a map $\LLogJ^\gl_{g,n} \to \LLogJ_{g,n}$ extending the pullback map on degree zero line bundles, but such a map does not exist (see for example \cite[Section~3.2]{kass2008lecture}). For comparing the obstruction theories however, it suffices to work with opens containing $\LLogDRL_{A,b}$.

One of the obstructions for the pullback map to extend is that $\theta$-stable line bundles on $C^\gl$ are not necessarily stable when pulled back to $C$. Tautologically, the pullback map does extend to $\theta$-stable line bundles that remain $\theta$-stable after pullback.

\begin{definition}
Let $\LLogJ_{g,n}^{\gl,\circ} \subset \LLogJ_{g,n}^\gl$ be the open parametrising $\theta$-stable line bundles on $C^\gl$ that are $\theta$-stable when pulled back to $C$, and let \[\gl^*: \LLogJ_{g,n}^{\gl,\circ} \to \LLogJ_{g,n}\] be the corresponding pullback map.

Let $I_{g,n} = \Mpt_{g,n} \times_{\LLogJ_{g,n}} \LLogJ_{g,n}^{\gl,\circ}$ be the space of $\theta$-stable line bundles on $C^\gl$ that pull back to the trivial line bundle.
\end{definition}

\begin{definition}
Write $\LLogDRL({A,b,-b})^\circ \subset \LLogDRL({A,b,-b})$ and $\Mpt_{g,n}^\circ \subset \Mpt_{g,n}$ for the logarithmic opens defined by the vanishing of $\delta_{A,b,-b}^\theta$. Write $\alpha_{A,b,-b}^{\gl,\theta}$ be the piecewise linear function $C^\gl/\Mpt_{g,n}^\circ$ obtained by gluing $\alpha_{A,b,-b}^\theta$.

Write $\sigma_b$ for the map
\begin{align*}
\Mpt_{g,n}^\circ &\to \LLogJ_{g,n}^{\gl} \\
C/S &\mapsto \omega_{C_{\gl}}^{\log,\tensor -k}(\alpha_{A,b,-b}^{\gl,\theta})
\end{align*}
\end{definition}

\begin{lemma}
The map $\sigma_b$ factors through the open $\LLogJ_{g,n}^{\gl,\circ}$. We have isomorphisms of log line bundles
\[
\sigma_b(C/S) \cong \omega_{C_{\gl}}^{\log,\tensor -k}(\sum_i a_i) \in \LLogJ_{g,n}^{\gl,\circ}
\]
and
\[
(\gl^* \circ \sigma_b)(C/S) \cong \omega_{C}^{\log,\tensor -k}(\sum_i a_i + b p_{n-1} - b p_n) \in \LLogJ_{g,n}.
\]
\end{lemma}
\begin{proof}
The piecewise linear function $\alpha_{A,b,-b}^{\gl,\theta}$ pulls back to $\alpha_{A,b,-b}^{\theta}$, so the line bundle $\Ocal_{C^\gl}(\alpha_{A,b,-b}^{\gl,\theta})$ pulls back to $\Ocal_C(\alpha_{A,b,-b}^{\theta})$. This is $\theta$-stable by \ref{def:alphabtheta}, hence $\Ocal_{C^\gl}(\alpha_{A,b,-b}^{\gl,\theta})$ lies inside $\LLogJ_{g,n}^{\gl,\circ}$. The statements on the isomorphisms of log line bundles is exactly \ref{lem:linebundlepl} for $B = A$ and $B = (A,b,-b)$ respectively. 
\end{proof}

\begin{remark}
The map $\sigma_b: \LLogDRL({A,b,-b})^\circ \to \LLogJ_{g,n}^{\gl,\circ}$ extends naturally to a map $\LLogDRL({A,b,-b}) \to \LLogJ_{g,n}^{\gl}$, given by gluing the $\G_\log$-torsor $P/C$ using the transition function $\delta_{A,b,-b}^\log: \LLogDRL({A,b,-b}) \to \G_\log$. This torsor has bounded monodromy, as the monodromy around the new loop comes from a piecewise linear function on the unglued curve. This map does not factor through the embedding $\LLogJ_{g,n}^{\gl,\circ} \subset \LLogJ_{g,n}^{\gl}$.
\end{remark}


All in all, we obtain the following commutative diagram

\begin{equation}
\label{eq:thediag}
\begin{tikzcd}
\LLogDRL_{A,b}^\gl \arrow[r] \arrow[d] & \LLogDRL({A,b,-b})^\circ \arrow[r] \arrow[d, "\sigma_b"] & \Mpt_{g,n}^\circ \arrow[d, "\sigma_b"]\\
\Mpt_{g,n} \arrow[r, "i"] & I_{g,n} \arrow[r, "e'"] \arrow[d, "\gl^*"] &\LLogJ^{\gl,\circ}_{g,n} \arrow[d, "\gl^*"] \\
& \Mpt_{g,n} \arrow[r,"e"]& \LLogJ_{g,n} 
\end{tikzcd}
\end{equation}
where all squares are fiber squares.

We remark that $i$ is a regular closed embedding of codimension $1$, and both $e$ and $e' = \gl^* e$ are regular closed embeddings of codimension $g$. This can be checked on the open of multidegree $0$ line bundles, where it follows from the short exact sequence
\[
0 \to \G_m \to \gl^* \Jac_{g+1,n-2}^{\ul{0}} \to \Jac_{g,n}^{\ul{0}} \to 0.
\]
and the fact that $\Jac_{g,n}^{\ul{0}}$ is smooth of dimension $g$.

\begin{proof}[Proof of \ref{prop:vfceq}]
The virtual fundamental class $[\LLogDRL_{A,b}^\gl]^\vfc$ is defined as the Gysin pullback $(e'i)^! [\Mpt_{g,n}^\circ]$, where $e',i$ are the maps in \ref{eq:thediag}. By functoriality we have $(e'i)^! = i^!e'^! $, and as $e'$ is the base change of $e$ and $e,e'$ are both regular closed embeddings of codimension $g$, we have $e'^! = e^!$.

Then note that $e^! [\Mpt_{g,n}^\circ] = [\LLogDRL({A,b,-b})^\circ]^\vfc$. It remains to show an equality of the two Gysin pullbacks
\[
i^!, j^! : \LogCH(\LLogDRL(A,b,-b)) \to \LogCH(\LLogDRL_{A,b}^\gl)
\] 
defined by the two fiber squares
\[
\begin{tikzcd}
\LLogDRL_{A,b}^\gl \arrow[r] \arrow[d] \pb & \LLogDRL_{A,b,-b}^\circ \arrow[d, "\sigma_b"]\\ 
\Mpt_{g,n} \arrow[r, "i"] & I_{g,n}
\end{tikzcd}
\qquad
\begin{tikzcd}
\LLogDRL_{A,b}^\gl \arrow[r] \arrow[d] \pb & \LLogDRL_{A,b,-b} \arrow[d, "\delta_{A,b,-b}^\log"]\\ 
* \arrow[r, "j"] & \G_\log
\end{tikzcd}
\]

Note that $i$ and $j$ are both regular embeddings of codimension $1$. It suffices to check compatibility of the pullback of the normal bundles of $i, j$. This follows from the commutativity of the following diagram of fiber squares
\[
\begin{tikzcd}
& \LLogDRL_{A,b}^\gl \arrow[r] \arrow[d] & \LLogDRL_{A,b}^\gl \times \G_m \arrow[d] &\\
& \Mpt_{g,n} \arrow[r] \arrow[rd] \arrow[ld] & \Mpt_{g,n} \times \G_m \arrow[rd] \arrow[ld] & \\
\Mpt_{g,n} \arrow[r, "i"] & I_{g,n} & * \arrow[r, "j"] & \G_\log 
\end{tikzcd}
\]
\end{proof}

\begin{remark}
In fact, we see that the virtual normal bundles $\sigma_b^* N_{\Mpt_{g,n}/I_{g,n}}$ and $\delta_{A,b,-b}^{\log,*} N_{*/\G_\log}$ for the map $\LLogDRL_{A,b}^\gl/\LLogDRL_{A,b}^\gl$ are both the trivial line bundle on $\LLogDRL_{A,b}^\gl$.
\end{remark}

\subsection{Proof of \ref{thm:firstformula}}
\label{subsec:proof}

We now can compute the class $\delta_{A,b,-b}^{\log,!} [\LLogDRL_{A,b,-b}]^\vfc$. 

\begin{lemma}
\label{lem:1eq0}
Write $\pi: \LLogDRL_{A,b}^\gl \to \Mpt_{g,n}$. Then we have
\begin{align*}
\pi_*(\delta_{A,b,-b}^{\log,!} [\LLogDRL_{A,b,-b}]^\vfc) &= \frac12 \LLogDR_{A,b,-b}\cdot \Phi(|\delta_{A,b,-b}^\theta|) \\
&= \LLogDR_{A,b,-b}\cdot \Phi(\max(\delta_{A,b,-b}^\theta,0)) \\
&= \LLogDR_{A,b,-b}\cdot \Phi(\max(-\delta_{A,b,-b}^\theta,0)).
\end{align*}
\end{lemma}
\begin{proof}
Write $\pi_1: \LLogDRL_{A,b}^\gl \to \LLogDRL_{A,b,-b}$, and $\pi_2: \LLogDRL_{A,b,-b} \to \Mpt_{g,n}$, so that $\pi = (\pi_2 \circ \pi_1)$.

Note that $[1] \in \P^1$ is equivalent to $\frac12([0] + [\infty]), [0]$ and $[\infty]$, and correspondingly we have $[1] = \frac12 \Phi(|x|) = \Phi(\max(x,0)) = \Phi(\max(-x,0))$ where $x$ is the universal PL function on $\G_\log$. Then $x$ pulls back to the piecewise linear function $\ol{\delta_\log} = \delta_{A,b,-b}^\theta$ on $\LLogDRL_{A,b,-b}$, so 
\begin{align*}
\pi_{1,*}\delta_{A,b,-b}^{\log,!} [\LLogDRL_{A,b,-b}]^\vfc = [\LLogDRL_{A,b,-b}]^\vfc \cdot y
\end{align*}
for any $y \in \{\frac12 \Phi(|x|), \Phi(\max(x,0)), \Phi(\max(-x,0))\}$. The lemma follows by the projection formula for $\pi_2$.
\end{proof}

\begin{proof}[Proof of \ref{thm:firstformula}]
By \ref{def:complog} the class $\gl^* \LLogDR(A)$ splits as a sum
\[
\sum_{b \in \Z} \LLogDR_{A,b}^\gl
\]
where the vanishing of the terms follows from \ref{lem:bounded}.

The formulas for $\LLogDR_{A,b}^\gl$ follow from \ref{prop:vfceq} and \ref{lem:1eq0}.
\end{proof}

\section{The pullback of the double ramification cycle}
\label{sec:pullbackDR}

In this section we will give a formula for the pushforward of $\gl^* \LLogDR(A)$ to $\CH(\Mbar_{g,n})$.

\begin{definition}
\label{def:banana}
A \emph{banana graph} is a connected graph $\Gamma$ of type $(g,n)$ with $V(\Gamma) = \{v_1,v_2\}$ such that all edges are incident to $v_1$ and $v_2$. We further require leg $p_{n-1}$ to be incident to $v_1$, leg $p_n$ to be incident to $v_2$. For $A \in \Z^n$ with sum $k(2g-2+n)$ and $b \in \Z$, we write $S_{A,b}$ for the groupoid
\[
  \{(\Gamma,B)\}
\]
where $\Gamma$ is a banana graph, and $B \in \Z^{E(\Gamma)}$ with sum $-b - \sum_{i \in L(v_1)} b_i + k(2g(v_1) - 2 + n(v_1))$. If $b \geq 0$, we require $b_e > 0$ and if $b < 0$ we require $b_e < 0$.

For $(\Gamma,B) \in S_{A,b}$, we write $C_1$ for $((a_i)_{i \in L(v_1)},b,(b_e)_{e \in E(\Gamma)}) \in \Z^{H(v_1)}$ and 
$C_2$ for $((a_i)_{i \in L(v_2)},-b,(-b_e)_{e \in E(\Gamma)}) \in \Z^{H(v_2)}$.  
\end{definition}

\begin{theorem}
\label{thm:secondformula}
We have
\begin{equation}
\begin{split}
\gl^* \DR(A) = \sum_{b} \sum_{(\Gamma,B) \in S_{A,b}} \frac{\prod_{e \in E(\Gamma)}|b_e|}{\Aut(\Gamma,B)} \zeta_{\Gamma,*}\left(\DR(C_1)\tensor\DR(C_2)\right).
\end{split}
\end{equation}
\end{theorem}
\begin{proof}
We fix a $b \in \Z$ and assume without loss of generality $b \geq 0$. We write $\tilde{\Mcal}_{g,n}$ and $\tilde{\Mpt}_{g,n}$ for the log blowups of $\Mbar_{g,n}$ and $\Mpt_{g,n}$ supporting $\LogDR(A,b,-b)$, defined in \cite[Section~4]{Holmes2022Logarithmic-double}.

Write $\delta = \delta_{A,b,-b}^\theta$, a strict piecewise linear function on $\tilde{\Mpt}_{g,n}^\trop$. We will compute $\LogDR(A,b,-b) \cdot \max(\delta,0)$.

Let $\pi$ denote the projection \[\tilde{\Mpt}_{g,n}^\trop = \tilde{\Mcal}_{g,n}^\trop \times \R_{\geq 0}^n \to \tilde{\Mcal}_{g,n}^\trop,\] and let $s$ denote the zero section of $\pi$. Let $\delta' = s^* \max(\delta,0)$ as a piecewise linear function on $\tilde{\Mcal}_{g,n}^\trop$. Note that $\delta$ is positive only on rays corresponding to banana graphs. Then $\max(\delta,0)$ and $\pi^* \delta'$ are piecewise linear functions that have the same values on rays, and hence $\max(\delta,0) = \pi^* \delta'$.

For clarity, we describe the piecewise linear function $\delta'$ on $\Mbar_{g,n}^\trop$. For a classical (marked, but not pierced) tropical curve $\Gamma \in \Mbar_{g,n}^\trop$, write $v_1,v_2$ for the vertices connected to $p_{n-1},p_n$ respectively. Write $\alpha = \alpha_{A,b,-b}^\theta$ for the piecewise linear function that $\theta$-stabilises $\Ocal(\sum_i a_i p_i + b p_{n-1} - b p_n)$. Then $\delta'(\Gamma) = \alpha(v_2) - \alpha(v_1)$.

For $(\Gamma,B) \in S_{A,b}$ write $\Mbar_{\Gamma,B}$ for the corresponding stratum in Note that on $\Mbar_{g,n}^\trop$, we then have that $\Phi(\delta')$ is a linear combination of the boundary divisors $[\Mbar_{\Gamma,B}]$, and it suffices to compute $[\Mbar_{\Gamma,B}] \cdot \LogDR(A,b,-b)$ with the correct multiplicity.

Write $\pi: \Mbar_{\Gamma,B} \to \Mbar_{g(v_1),n(v_1)} \times \Mbar_{g(v_2),n(v_2)}$ for the map to the stratum of $\Gamma$. This induces a finite map \begin{equation}\label{eq:mubtorsor}\pi: \LogDRL(A,b,-b) \cap \Mbar_{\Gamma,B} \to \LogDRL(C_1) \times \LogDRL(C_2).\end{equation} This map is a kummer log \'etale $\mu_B$-torsor, where we write $\mu_B = \prod_{i} \mu_{b_i}$. Over a geometric point in the open stratum, this $\mu_B$-torsor is isomorphic to the log scheme $(\delta')^{1/B} = \{(x_e)_{e\in E} : \forall e, x_e^{b_e} = \delta'\} \subset \A^E$ over the log point $(\pt, \N \cdot \delta')$.
Note that the piecewise polynomial $\delta'$ induces the class of degree $\prod_i b_i$ on $(\delta')^{1/B}$ (see \ref{rem:mubtorsor} for an example on what this torsor looks like). All in all, the virtual class on $\LogDRL(A,b,-b) \cap \Mbar_{\Gamma,B}$ pushes forward to
\[
  \Aut(\Gamma,B)^{-1} [\Gamma,\prod_e b_e \DR(C_1) \DR(C_2)].
\]
\end{proof}

\begin{remark}
\label{rem:mubtorsor}
Note that the kummer log \'etale $\mu_B$-torsor \ref{eq:mubtorsor} is always of log degree $\prod_{e} b_e$, and after taking an appropriate root stack is of geometric degree $\prod_e b_e$. The actual geometric structure before taking a root stack varies, depending on the gcds of the slopes $b_e$. Over the category of coherent log schemes (log schemes with charts that are not necessarily integral or saturated) the functor $(\delta')^{1/B} = \{(x_e)_{e\in E} : \forall e, x_e^{b_e} = \delta'\} \subset \A^E$ over $(\pt, \N \cdot \delta')$ is represented by $X = (\pt, M)$ where $M$ is the non-integral monoid $(\N \cdot \delta' \oplus \bigoplus_e \N \cdot \ell_e)/(b_e \ell_e = \delta')$, and the piecewise linear function $\delta'$ gives rise to the class $\prod_e b_e [X]$ of degree $\prod_e b_e$.

Over the category of fs log schemes, the same functor is represented by the fs-ification $Y$ of $X$. This has characteristic monoid $\N \cdot \frac{1}{m} \delta'$ where $m = \lcm(b_e)_{e \in E}$, and the underlying scheme $\ul{Y}$ is of degree $\frac{1}{m}\prod_{e} b_e$. Hence, the piecewise linear function $\delta'$ gives rise to the class $m[Y] \in \CH_*(Y)$ on $Y$. Again, this class has degree $\prod_e b_e$.

For example, take $B = (5,5)$. Then $Y$ has functor $\{(x,y) \in \A^2 : x^5 = y^5 = \delta'\}$, and by integrality this is isomorphic to $\{(x,z) \in \A^2 : x^5 = \delta', z^5 = 1\}$, representable by $(\pt,\N \cdot 1/5 \delta') \times \mu_5$.
\end{remark}

In the above proof, we use the equality $\LogDR_A,b^\gl = \LogDR(A,b,-b) \cdot \max(\delta_{A,b,-b}^{\theta},0)$ from \ref{thm:firstformula} for $b \geq 0$ and $\LogDR_A,b^\gl = \LogDR(A,b,-b) \cdot \max(-\delta_{A,b,-b}^{\theta},0)$ for $b < 0$. The equality $\LogDR(A,b,-b) \cdot \max(\delta_{A,b,-b}^{\theta},0) = \LogDR(A,b,-b) \cdot \max(-\delta_{A,b,-b}^{\theta},0)$ gives rise to a special case of the following tautological relation, first proven for $k = 0$ in \cite[Section~2]{zvonkine2015IntegralsPsiClasses} and for general $k$ in \cite[Proposition~3.1]{costantini2025integralspsiclassestwisteddouble}. The proof in \cite{costantini2025integralspsiclassestwisteddouble} is through strata of differentials and intricate combinatorics, proving it for a dense set of inputs and invoking polynomiality \cite{Pixton2023DRPoly,spelier2024polynomiality}. The theory of pierced curves allows us to give a very short proof.


\begin{proposition}
\label{prop:splitpsi}
Fix $A \in \Z^n$ with sum $k(2g-2+n)$. Then we have
\[
(a_{n-1}\psi_{n-1} - a_n \psi_n) \DR(A) + \sum_{(\Gamma,B) \in S}  \frac{s(\Gamma)}{\Aut(\Gamma,B)} \prod_{e \in E(\Gamma)} |b_e| [\Gamma,\DR(C_1) \tensor \DR(C_2)] = 0
\]
where the sum is over the set $S_{a,b} = \{(\Gamma,B)\}$ where
\begin{itemize}
  \item $\Gamma$ is a banana graph with vertices $v_1,v_2$;
  \item $B \in \Z_{> 0}^E$ with sum $- \sum_{i \in L(v_1)} a_i + k(2g(v_1) - 2 + n(v_1))$;
\end{itemize}
and \[s(\Gamma) = \begin{cases}  -1 &\text{ if } p_{n-1} \in L(v_1) \wedge p_n \in L(v_2) \\ 1 &\text{ if } p_{n-1} \in L(v_2) \wedge p_n \in L(v_1) \\ 0 &\text{ else}\end{cases}\]
and we write $C_1$ for $((a_i)_{i \in L(v_1)},(b_e)_{e \in E(\Gamma)}) \in \Z^{H(v_1)}$ and $C_2$ for $((a_i)_{i \in L(v_1)},(-b_e)_{e \in E(\Gamma)}) \in \Z^{H(v_2)}$.
\end{proposition}


\begin{proof}[Proof of \ref{prop:splitpsi}]
Write $\alpha = \alpha_A^\theta$ for the piecewise linear polynomial defined in \ref{def:alphabtheta}. Let $\delta: \Mpt_{g,n}^\trop \to \G_\trop$ be given by $\alpha(p_n) - \alpha(p_{n-1})$.
We will prove the desired equality by computing $\LLogDR(A) \cdot \Phi(\delta)$ in two different ways.

On $\LLogDR(A)$, the piecewise linear function $\delta$ lifts to a map $\LLogDR(A) \to \G_\log$ given by $\ev_{n}/\ev_{n-1}$ where $\ev_i$ are the evaluation maps as defined in \ref{def:evmaps}. The canonical piecewise linear function $x: \G_\log \to \G_\trop$ satisfies $\Phi(x) = 0$, so $\LLogDR(A) \cdot \Phi(\delta) = 0$.

We can also compute $\LLogDR(A) \cdot \Phi(\delta)$ by splitting $\delta$ as a sum over the rays in $\tilde{\Mpt}_{g,n}^\trop = \tilde{\Mcal}_{g,n}^\trop \times \R_{\geq 0}^n$. Note that $\Phi(\delta)$ splits as a sum over the banana graphs, plus $\Phi(a_{n} \ell_n - a_{n-1} \ell_{n-1})$. The banana graph $(\Gamma,B)$ contributes $s(\Gamma) \prod_{e \in E(\Gamma)} b_e [\Gamma,\DR(C_1) \tensor \DR(C_2)]$, and by \cite[Lemma~4.3]{Holmes2023LogarithmicCohomologicalFT} we have $\Phi(a_{n} \ell_n - a_{n-1} \ell_{n-1}) = a_{n-1}\psi_{n-1} - a_n \psi_n$.
\end{proof}

\bibliographystyle{alpha}
\addcontentsline{toc}{section}{References}
\bibliography{prebib.bib}


\end{document}